\documentclass[10pt]{amsart}
\usepackage{amsfonts}

\usepackage[centertags]{amsmath}
\usepackage{color,graphicx,shortvrb}
\usepackage{amsthm}
\usepackage{newlfont}
\usepackage{graphicx}
\usepackage{amsmath}
\usepackage{epsfig}

\newtheorem{theorem}{Theorem}

\newtheorem{corollary}[theorem]{Corollary}

\newtheorem{lemma}[theorem]{Lemma}

\newtheorem{remark}[theorem]{Remark}

\begin{document}
\title{Stability of a tree-shaped network of strings and beams}
\author{Ka\"{\i}s AMMARI}
\address{UR Analysis and Control of Pde, UR 13ES64, Department of Mathematics,
Faculty of Sciences of Monastir, University of Monastir, 5019 Monastir, Tunisia}
\email{kais.ammari@fsm.rnu.tn}
\author{Farhat Shel}
\address{UR Analysis and Control of Pde, UR 13ES64, Department of Mathematics,
Faculty of Sciences of Monastir, University of Monastir, 5019 Monastir, Tunisia,
Tunisia}
\email{farhat.shel@ipeit.rnu.tn}

\begin{abstract}
In this paper we study the stability of a tree-shaped network of elastic
strings and beams with some feedbacks at the ends. The whole system is
asymtotically stable. Moreover, the energy of the solution decay
exponentially to zero if there is no beam following a string ( from the root
to the leaves) and decay polynomially if not. We use a frequency domain
method from the semigroups theory.
\end{abstract}

\subjclass[2010]{35R02, 35M33, 35B40, 93D15, 74K05, 74K10}
\keywords{uniform stability, polynomial stability , frequency domain method, network,
strings, Euler-Bernoulli beams}
\maketitle
\tableofcontents

\thispagestyle{empty}

\vfill\break

\section{Introduction}

In past decades, the dynamic behavior of networks of flexible structures
have been studied by some authors. See for instance, \cite{Lag94, Zua06, KSL}
and the references therein. The importance of these studies lies in the need
for engineering to eliminate vibrations in such composite structures. In
this paper we consider a model $(\mathcal{S})$ of a tree-shaped network of $%
N $ elastic materials, constituting of \textbf{strings} and Euler-Bernoulli 
\textbf{beams}.

To start let us first introduce some notations for the tree under
consideration (as introduced in \cite{Abd12} or in \cite{Mer08}). Let $%
\mathcal{G}$ be a finite planar tree, with $N$ edges and $p=N+1$ vertices.
We denote by $E=\{e_{1},...,e_{N}\}$ the set of edges, $\mathcal{V}%
=\{a_{1},...,a_{p}\}$, the set of vertices of $\mathcal{G}$ and we suppose
that $a_{1}$ is the root of $\mathcal{G}$ and that $a_{1}$ and $a_{2}$ are
ends of $e_{1}.$

\noindent For a fixed vertex $a_{k},$ let $I(a_{k})$ be the set of indices of
edges adjacent to $a_{k}$. Denote by $\mathcal{V}_{int}$ the set of interior
vertices and $\mathcal{V}_{ext}$ the set of exterior vertices of $\mathcal{G}%
.$ Each edge $e_{j}$ of length $\ell _{j}$ is a curve parametrized by 
\begin{equation*}
\pi _{j}:[0,\ell _{j}]\longrightarrow e_{j},\;x_{j}\mapsto \pi _{j}(x_{j}),
\end{equation*}
and sometimes identified with the interval $(0,\ell _{j}).$ The incidence
matrix $D=(d_{ij})_{p\times N}$ is defined by 
\begin{equation*}
d_{ij}=\left\{ 
\begin{tabular}{l}
$1$ if $\pi _{j}(\ell _{j})=a_{i},$ \\ 
$-1$ if $\pi _{j}(0)=a_{i},$ \\ 
$0$ otherwise,
\end{tabular}
\right.
\end{equation*}
and for a function $u:\mathcal{G}\longrightarrow \mathbb{C}$ we set $%
u^{j}=u\circ \pi _{j}$ its restriction to the edge $e_{j}$ and we will
denote $u^{j}(x)=u^{j}(\pi _{j}(x))$ for any $x$ in $(0,\ell _{j}).$

We denote by $\left\langle .,.\right\rangle $ and $\left\| .\right\| $ the
inner product and norm in $L^{2}$-space, respectively.

Suppose that the equilibrium position of the tree of elastic strings and
beams coincides with $\mathcal{G}.$

Our model is then described as follows: every string $e_{j}$ satisfies the following equation 
\begin{equation}
u_{tt}^{j}-u_{xx}^{j}=0\text{ in }(0,\ell _{j})\times (0,\infty ),  \label{a}
\end{equation}
and every beam $e_{j}$ satisfies the following equation 
\begin{equation}
u_{tt}^{j}+u_{xxxx}^{j}=0\text{ in }(0,\ell _{j})\times (0,\infty )
\label{b}
\end{equation}
where $u^{j}=u^{j}(x,t)$ is the function describing the displacement of the
string or beam $e_{j}$.

\noindent The initial conditions are 
\begin{equation}
u^{j}(x,0)=u_{0}^{j}(x),\;\;u_{t}^{j}(x,0)=u_{1}^{j}(x).  \label{c}
\end{equation}

\noindent Denote by $I_{S}(a_{k})$ (resp. $I_{B}(a_{k})$) the set of strings
(resp. beams) adjacent to $a_{k}$ and by $\mathcal{V}_{ext}^{S}$ and $%
\mathcal{V}_{ext}^{B}$ respectively the set of external nodes of strings and
those of beams, different from $a_{1}$. Then the transmission conditions at
the inner nodes are 
\begin{equation}
\left\{ 
\begin{tabular}{l}
$u^{j}(a_{k},t)=u^{l}(a_{k},t),\ \;\;\;\;j,l\in I(a_{k}),\;\;\;\;a_{k}\in 
\mathcal{V}_{int},$ \\ 
$u_{xx}^{j}(a_{k},t)=u_{xx}^{l}(a_{k},t),\ \;\;\;\;j,l\in
I_{B}(a_{k}),\;\;\;\;a_{k}\in \mathcal{V}_{int},$ \\ 
$\sum\limits_{j\in
I_{B}(a_{k})}d_{kj}u_{x}^{j}(a_{k},t)=0,\;\;\;\;\;a_{k}\in \mathcal{V}_{int}$%
, \\ 
$\sum\limits_{j\in I_{B}(a_{k})}d_{kj}u_{xxx}^{j}(0,t)-\sum\limits_{j\in
I_{S}(a_{k})}d_{kj}u_{x}^{j}(0,t)=0,\;\;a_{k}\in \mathcal{V}_{int},$%
\end{tabular}
\right.  \label{d}
\end{equation}
and the boundary conditions are 
\begin{equation}
\left\{ 
\begin{tabular}{l}
$u^{j_{k}}(a_{k},t)=0,\;\;\;\;\;a_{k}\in \mathcal{V}_{ext},$ \\ 
$u_{xx}^{j_{k}}(a_{k},t)=0,\;\;\;\;a_{k}\in \mathcal{V}_{ext}^{B}$%
\end{tabular}
\right. \text{{}}  \label{boun1}
\end{equation}
where $j_{k}$ is the index of the unique edge adjacent to $a_{k}\in \mathcal{%
V}_{ext}.$

For a classical solution $u$ of $(\mathcal{S}),$ the energy is defined as
the sum of the energy of its components, that is, 
\begin{eqnarray*}
E(t) &=&\frac{1}{2}\sum\limits_{j=1}^{N}\int_{0}^{\ell _{j}}\left|
u_{t}^{j}(x,t)\right| ^{2}dx+\frac{1}{2}\sum\limits_{j\in
I_{S}}\int_{0}^{\ell _{j}}\left| u_{x}^{j}(x,t)\right| ^{2}dx \\
&&+\frac{1}{2}\sum\limits_{j\in I_{B}}\int_{0}^{\ell _{j}}\left|
u_{xx}^{j}(x,t)\right| ^{2}dx
\end{eqnarray*}
where $I_{S}$ and $I_{B}$ are the sets of indices of strings and beams
respectively.

\noindent Differentiate formally the energy function with respect to time $t,
$ we get 
\begin{equation*}
\frac{dE}{dt}(t)=0,
\end{equation*}
and the system is conservative.

Stability of such models of networks of
strings or of beams, has been proved before, by applying a control at an external node or by forcing
the damping conditions at inner nodes. In \cite{Amm04}, the authors proved
the polynomial stability of a star-shaped network of strings when a feedback
is applied at the common node and in \cite{Amm05} and \cite{Zua09} the
authors proved a similar result for a tree of strings when the feedback is
applied at an external node. In \cite{Amm07} we consider a network of beams.
See also \cite{Wan08} for exponential stability of a star-shaped network of
beams and \cite{Han10} for asymptotic stability of a star-shaped network of
Timoshenko beams. In \cite{Far12} and \cite{Far13} we add thermoelastic
edges to the network of elastic materials to obtain an exponential stability
result.

For strings-beams networks see \cite{Amm11} where the authors considered a
star-shaped network of beams and a string, with controls applied at all the
exterior nodes. They proved a result of exponential stability. Some results
of polynomial stability have proved before, for coupled string-beam systems 
\cite{Amm09} (see \cite{KSL} for general setting) and for chains of
alternated beams and strings \cite{Amm12}, when feedbacks are applied at
inner nodes. The case of a $2-d$ coupled system of a wave equation and a
plate equation has been studied by K. Ammari and S. Nicaise in \cite{Amm10}.
They proved a result of exponential stability under some geometric
conditions.

In this paper we study a more general case of networks, in fact, it is the
model $(\mathcal{S})$ presented at the beginning, stabilized by applying
feedbacks at all leaves (the root remains free). For this, let $\mathcal{V}%
^{\ast }=\mathcal{V}-\{a_{1}\}$, $\mathcal{V}_{ext}^{\ast }=\mathcal{V}%
_{ext}-\{a_{1}\}$ and let $\delta $ in $\{0,1\}$ with $\delta =1$ if $e_{1}$
is a string and $\delta =0$ if $e_{1}$ is a beam. Then instead of (\ref
{boun1}) we take 
\begin{eqnarray}
u^{1}(a_{1},t) &=&0,\;(1-\delta )u_{xx}^{1}(a_{1},t)=0,  \notag \\
u_{x}^{j_{k}}(a_{k},t) &=&-d_{kj}u_{t}^{j_{k}}(a_{k},t),\;\;\;\;a_{k}\in 
\mathcal{V}_{ext}^{S},  \notag \\
u_{xxx}^{j_{k}}(a_{k},t)
&=&d_{kj}u_{t}^{j_{k}}(a_{k},t),\;\;u_{x}^{j_{k}}(a_{k},t)=0,\;\;\;\;a_{k}%
\in \mathcal{V}_{ext}^{B}.  \notag
\end{eqnarray}
Formally, we have 
\begin{equation*}
\frac{d}{dt}E(t)=-\sum\limits_{a_{k}\in \mathcal{V}_{ext}^{\ast }}\left|
u_{t}^{j_{k}}(a_{k},t)\right| ^{2}\leq 0.
\end{equation*}
So the system is dissipative.

\noindent We prove different decay results of the energy of the system
depending on the position of beams relative to strings. Precisely $(\mathcal{%
S})$ is exponentially stable if there is no beam following a string from the
root to leaves and polynomially stable if not. Moreover, we give an example
corresponding to the last case which is not exponentially stable. The method
that we use to show exponential or polynomial stability is based on the
resolvent approach.

The paper is organized as follows: In section \ref{sec2}, we reformulate the
system $(\mathcal{S})$ as an evolution equation in a Hilbert space, and
prove that it is associated with a $\mathcal{C}_{0}$-semigroup of
contraction and in section \ref{sec3}, by using frequency domain method, we
first prove, under some conditions, that the system $(\mathcal{S})$ is
exponentially stable, then we give a result of polynomial stability.

\section{Abstract setting} \label{sec2}

The aim of this section is to rewrite the system $(\mathcal{S})$ as an
evolution equation in an appropriate Hilbert space. We then prove the
existence and uniqueness of solutions of the problem using semigroup theory.

Let us consider 
\begin{equation*}
V=\left\{ f=(f_{S},f_{B})\in \prod\limits_{j\in I_{S}}H^{1}(0,\ell
_{j})\times \prod\limits_{j\in I_{B}}H^{2}(0,\ell _{j})\mid f\text{
satisfies (\ref{s1'})}\right\}
\end{equation*}
where  $f_{S}=(f_{j})_{j\in I_{S}}$, $f_{B}=(f_{j})_{j\in I_{B}},$ 
and 
\begin{equation}
\left\{ 
\begin{tabular}{l}
$f^{1}(a_{1})=0,$ \\ 
$f^{j}(a_{k})=f^{l}(a_{k}),\;\ j,l\in I(a_{k}),\;\;a_{k}\in \mathcal{V}%
_{int}, $ \\ 
$\partial _{x}f^{j_{k}}(a_{k})=0,\ \;\;a_{k}\in \mathcal{V}_{ext}^{B}$, \\ 
$\sum_{j\in I_{B}(a_{k})}d_{kj}\partial _{x}f^{j}(a_{k})=0,\ \;\;a_{k}\in 
\mathcal{V}_{int}$.
\end{tabular}
\right.  \label{s1'}
\end{equation}

\noindent Note that we can rewrite the last two equations in one, as follows 
\begin{equation*}
\sum_{j\in I_{B}(a_{k})}d_{kj}\partial _{x}f^{j}(a_{k})=0,\ \;\;a_{k}\in 
\mathcal{V}^{\ast }.
\end{equation*}

\noindent Define the energy space of $(\mathcal{S})$ by 
\begin{equation*}
\mathcal{H}=V\times \prod\limits_{j=1}^{N}L^{2}(0,\ell _{j})
\end{equation*}
endowed by the inner product 
\begin{equation*}
\left\langle y_{1},y_{2}\right\rangle _{\mathcal{H}}:=\sum_{j\in
I_{S}}\left\langle \partial _{x}f_{1}^{j},\partial
_{x}f_{2}^{j}\right\rangle +\sum_{j\in I_{B}}\left\langle \partial
_{x}^{2}f_{1}^{j},\partial _{x}^{2}f_{2}^{j}\right\rangle
+\sum_{j=1}^{N}\left\langle g_{1}^{j},g_{2}^{j}\right\rangle
\end{equation*}
where $y_{k}=\left( f_{k},g_{k}\right) ,$ $k=1,2.$ Then $\mathcal{H}$ is a
Hilbert space.

Now define the operator $\mathcal{A}$ on $\mathcal{H}$ by 
\begin{equation*}
\mathcal{D}(\mathcal{A})=\left\{ 
\begin{array}{c}
y=(u,v)\in V\times V\mid u_{S}\in \prod\limits_{j\in I_{S}}H^{2}(0,\ell _{j})%
\text{, }u_{B}\in \prod\limits_{j\in I_{B}}H^{4}(0,\ell _{j}) \\ 
\text{ and }y\text{ satisfies (\ref{s2''})}
\end{array}
\right\}
\end{equation*}
where 
\begin{equation}
\left\{ 
\begin{tabular}{l}
$\partial _{x}u^{j_{k}}(a_{k})=-d_{kj}v^{j_{k}}(a_{k}),\;\;a_{k}\in \mathcal{%
V}_{ext}^{S},$ \\ 
$(1-\delta )\partial _{x}^{2}u^{1}(a_{1})=0,$ \\ 
$\partial _{x}^{2}u^{j}(a_{k})=\partial
_{x}^{2}u^{l}(a_{k}),\;\;j,l=I_{B}(a_{k}),\;\;a_{k}\in \mathcal{V}_{int},$
\\ 
$\partial _{x}^{3}u^{j_{k}}(a_{k})=d_{kj}v^{j_{k}}(a_{k}),\;\;a_{k}\in 
\mathcal{V}_{ext}^{B},$ \\ 
$\sum\limits_{j\in I_{S}(a_{k})}d_{kj}\partial
_{x}u^{j}(a_{k})-\sum\limits_{j\in I_{B}(a_{k})}d_{kj}\partial
_{x}^{3}u^{j}(a_{k})=0,\;\;a_{k}\in \mathcal{V}_{int}$%
\end{tabular}
\right.  \label{s2''}
\end{equation}
and

\begin{equation*}
\mathcal{A}\left( 
\begin{array}{c}
u_{S} \\ 
u_{B} \\ 
v_{S} \\ 
v_{B}
\end{array}
\right) =\left( 
\begin{array}{c}
v_{S} \\ 
v_{B} \\ 
\partial _{x}^{2}u_{S} \\ 
-\partial _{x}^{4}u_{B}
\end{array}
\right), \, (u_S,u_B,v_S,v_B) \in {}(\mathcal{A}).
\end{equation*}

\noindent Then, putting $y=(u,u_{t}),$ we write the system $(\mathcal{S})$
into the following first order evolution equation 
\begin{equation}
\left\{ 
\begin{array}{c}
\frac{dy}{dt} = \mathcal{A}y, \\ 
y(0)=y_{0}
\end{array}
\right.  \label{0500}
\end{equation}
on the energy space $\mathcal{H}$, where $y_{0}=(u_{0},u_{1}).$

We have the following result,

\begin{lemma}
The operator $\mathcal{A}$ is the infinitesimal generator of a $\mathcal{C}%
_{0}$-semigroup of contraction $(T(t))_{t \geq 0}$.
\end{lemma}

\begin{proof}
By Lumer-Phillips' theorem (see \cite{Paz83}), it suffices to show that $\mathcal{A}$ is dissipative maximal. 
First, for any $y \in \mathcal{H}$ we have,
\begin{eqnarray*}
Re(\left\langle \mathcal{A}y,y\right\rangle _{\mathcal{H}}) &=&Re\left(
\sum_{j\in I_{S}}\int_{0}^{\ell _{j}}(\partial _{x}v^{j}\overline{\partial
_{x}u^{j}}dx+\partial _{x}^{2}u^{j}\overline{v^{j}})dx\right.  \\
&&+\left. \sum_{j\in I_{B}}\int_{0}^{\ell _{j}}(\partial _{x}^{2}v^{j}%
\overline{\partial _{x}^{2}u^{j}}dx-\partial _{x}^{4}u^{j}\overline{v^{j}}%
)dx\right) .
\end{eqnarray*}
 By somes integrations by parts, we obtain using the boundary and transmissions conditions (\ref{s1'}-\ref{s2''}), 
\begin{equation*}
Re(\left\langle \mathcal{A}y,y\right\rangle _{\mathcal{H}})=-\sum%
\limits_{a_{k}\in \mathcal{V}_{ext}^{\ast }}\left| v^{j_{k}}(\ell
_{j_{k}})\right| ^{2}\leq 0.
\end{equation*}
Then the operator $\mathcal{A}$ is dissipative.

We show now that every positive real number $\lambda $ belongs to $\rho (%
\mathcal{A}).\ $Let $z=(f,g)\in \mathcal{H},$ we look for $y=(u,v)\in 
\mathcal{D}(\mathcal{A})$ such that 
\begin{equation*}
(\lambda -\mathcal{A})y=z 
\end{equation*}
i.e.,
\begin{eqnarray}
\lambda u^{j}-v^{j} &=&f^{j},\;j=1,...,N, \\
\lambda v^{j}-\partial _{x}^{2}u^{j} &=&g^{j},\;j\in I_{S},  \label{lax2} \\
\lambda v^{j}+\partial _{x}^{4}u^{j} &=&g^{j},\;j\in I_{B}.  \label{lax3}
\end{eqnarray}
Then
\begin{eqnarray}
\lambda^{2} u^{j}-\partial _{x}^{2}u^{j} &=&g^{j}+\lambda f^{j},\;j\in I_{S},  \label{lax2'} \\
\lambda^{2} u^{j}+\partial _{x}^{4}u^{j} &=&g^{j}+\lambda f^{j},\;j\in I_{B}.  \label{lax3'}
\end{eqnarray}
Let $w$ in $V.$ Multiplying the first equation by $w_{S}$ and the second
equation by $w_{B}$, and integrating by parts, we get respectively 
\begin{equation*}
\lambda ^{2}\int_{0}^{\ell _{j}}u^{j}\overline{w^{j}}dx+\int_{0}^{\ell
_{j}}\partial _{x}u^{j}\partial _{x}\overline{w^{j}}dx-\left. \partial
_{x}u^{j}\overline{w^{j}}\right| _{0}^{\ell _{j}}=\int_{0}^{\ell
_{j}}(g^{j}+\lambda f^{j})\overline{w^{j}}dx
\end{equation*}
for $j$ in $I_{S}$ and 
\begin{equation*}
\lambda ^{2}\int_{0}^{\ell _{j}}u^{j}\overline{w^{j}}dx+\int_{0}^{\ell
_{j}}\partial _{x}^{2}u^{j}\partial _{x}^{2}\overline{w^{j}}dx+\left.
\partial _{x}^{3}u^{j}\overline{w^{j}}\right| _{0}^{\ell _{j}}-\left.
\partial _{x}^{2}u^{j}\partial _{x}\overline{w^{j}}\right| _{0}^{\ell
_{j}}=\int_{0}^{\ell _{j}}(g^{j}+\lambda f^{j})\overline{w^{j}}dx
\end{equation*}
for $j$ in $I_{B}.$ Now summing the two obtained equations, the left hand
side will be 
\begin{eqnarray*}
&&\lambda ^{2}\sum\limits_{j=1}^{N}\int_{0}^{\ell _{j}}u^{j}\overline{w^{j}}%
dx+\sum\limits_{j\in I_{S}}\int_{0}^{\ell _{j}}\partial _{x}u^{j}\partial
_{x}\overline{w^{j}}dx+\sum\limits_{j\in I_{B}}\int_{0}^{\ell _{j}}\partial
_{x}^{2}u^{j}\partial _{x}^{2}\overline{w^{j}}dx \\
&&+\sum\limits_{a_{k}\in \mathcal{V}_{int}}\left( -\sum\limits_{j\in
I_{S}(a_{k})}d_{kj}\overline{w^{j}}(a_{k})\partial
_{x}u^{j}(a_{k})+\sum\limits_{j\in I_{B}(a_{k})}d_{kj}\overline{w^{j}}%
(a_{k})\partial _{x}^{3}u^{j}(a_{k})\right)  \\
&&+ \sum\limits_{a_{k}\in \mathcal{V}_{ext}}v^{j_{k}}(a_{k})\overline{%
w^{j_{k}}}(a_{k})+\sum\limits_{a_{k}\in \mathcal{V}_{int}}\sum\limits_{j\in
I_{B}(a_{k})}d_{kj}\partial _{x}^{2}u^{j}(a_{k})\partial _{x}\overline{w^{j}}%
(a_{k}).
\end{eqnarray*}
We find, by taking into account (\ref{s1'}) and (\ref{s2''}) 
\begin{equation}
a(u,w)=F(w) \label{pb}
\end{equation}
where 
\begin{eqnarray*}
a(u,w) &=&\lambda ^{2}\sum\limits_{j=1}^{N}\int_{0}^{\ell _{j}}u^{j}%
\overline{w^{j}}dx+\sum\limits_{j\in I_{S}}\int_{0}^{\ell _{j}}\partial
_{x}u^{j}\partial _{x}\overline{w^{j}}dx+\sum\limits_{j\in
I_{B}}\int_{0}^{\ell _{j}}\partial _{x}^{2}u^{j}\partial _{x}^{2}\overline{%
w^{j}}dx \\
&&+\lambda \sum\limits_{a_{k}\in \mathcal{V}_{ext}}u^{j_{k}}(a_{k})\overline{%
w^{j_{k}}}(a_{k})
\end{eqnarray*}
and 
\begin{equation*}
F(w)=\sum\limits_{j=1}^{N}\int_{0}^{\ell _{j}}(g^{j}+\lambda f^{j})\overline{%
w^{j}}dx+\lambda \sum\limits_{a_{k}\in \mathcal{V}_{ext}}f^{j_{k}}(a_{k})%
\overline{w^{j_{k}}}(a_{k}).
\end{equation*}
$a$ is a continuous sesquilinear form on $V\times V$ and $F$ is a continuous
anti-linear form on $V.$ Moreover, there exists $C>0$ such that, for every $w\in V$%
\begin{equation*}
\left| a(w,w)\right| \geq C\left\| w\right\| _{V}^{2},
\end{equation*}
where
\begin{equation*}
\left\| u\right\| _{V}^{2}=\sum_{j=1}^{N}\int_{0}^{\ell _{j}}\left|
u^{j}\right| ^{2}dx+\sum_{j\in I_{S}}\int_{0}^{\ell _{j}}\left| \partial
_{x}u^{j}\right| ^{2}dx+\sum_{j\in I_{B}}\int_{0}^{\ell _{j}}\left| \partial
_{x}^{2}u^{j}\right| ^{2}dx.
\end{equation*}
By the Lax-Milgram's lemma, problem (\ref{pb}) has a unique solution $u$ in $%
V.$ It is easy to verify that: $u$ belongs to $\prod\limits_{j\in
I_{S}}H^{2}(0,\ell _{j})\times \prod\limits_{j\in I_{B}}H^{4}(0,\ell _{j}),$ $%
v=\lambda u-f\in V,$ $u_{S}$ and $u_{B}$ satifies respectively (\ref{lax2})
and (\ref{lax3}), and the conditions 
\begin{equation*}
\left\{ 
\begin{tabular}{l}
$\partial _{x}u^{j_{k}}(a_{k})=-d_{kj}v^{j_{k}}(a_{k}),\;\;a_{k}\in \mathcal{V}%
_{ext}^{S},$ \\ 
$(1-\delta )\partial _{x}^{2}u^{1}(\ell _{1})=0,$ \\ 
$\partial _{x}^{2}u^{j}(a_{k})=\partial
_{x}^{2}u^{l}(a_{k}),\;\;j,l=I_{B}(a_{k}),\;\;a_{k}\in \mathcal{V}_{int},$
\\ 
$\partial _{x}^{3}u^{j_{k}}(a_{k})=d_{kj}v^{j_{k}}(a_{k}),\;\;a_{k}\in \mathcal{V}%
_{ext}^{B},$ \\ 
$\sum\limits_{j\in I_{S}(a_{k})}d_{kj}\partial _{x}u^{j}(a_{k})-\sum\limits_{j\in
I_{B}(a_{k})}d_{kj}\partial _{x}^{3}u^{j}(a_{k})=0,\;\;a_{k}\in \mathcal{V}_{int}.$%
\end{tabular}
\right. 
\end{equation*}
Furthermore 
\begin{equation*}
\left\| y\right\| _{\mathcal{H}}^{2}\leq c\left\| z\right\| _{\mathcal{H}%
}^{2},
\end{equation*}
where $c$ is a positive constant independent of $y.$ In conclusion, $%
y=(u,v)\in \mathcal{D}(\mathcal{A})$ and $(\lambda -\mathcal{A})^{-1}\in 
\mathcal{L}(\mathcal{H}),$ that is, $\lambda \in \rho (A).$
\end{proof}

\begin{corollary}
For an initial datum $y_{0}\in \mathcal{H}$ there exists a unique solution $%
y\in C([0,+\infty ),\mathcal{H})$ of problem (\ref{0500}). Moreover if $%
y_{0}\in \mathcal{D}(\mathcal{A}),$ then $y\in C([0,+\infty ),\mathcal{D}(%
\mathcal{A}))\cap C^{1}([0,+\infty ),\mathcal{H}).$
\end{corollary}

\begin{remark}
\label{remarque}Note that, by the Sobolev embedding theorem, $(I-\mathcal{A}%
)^{-1}$ is a compact operator and then the spectrum of $\mathcal{A}$
consists of all isolated eigenvalues, i.e., $\sigma (\mathcal{A})=\sigma
_{p}(\mathcal{A}).$
\end{remark}

\section{Asymptotic behavior}

\label{sec3}

The aim of this section is to show that the system $(\mathcal{S})$ is
asymptotically stable. Moreover, we will prove that the solution is
exponentially stable if there is no beam following a string, from the root
to the leaves, as in the first tree (Figure \ref{fig1}), and polynomially
stable if at least a beam follows a string, as in the second tree (Figure 
\ref{fig2}). Finally, The lack of exponential stability is proved on an
example.

\begin{figure}[tbp]
\begin{minipage}[b]{.48\linewidth}
  \centering\epsfig{figure=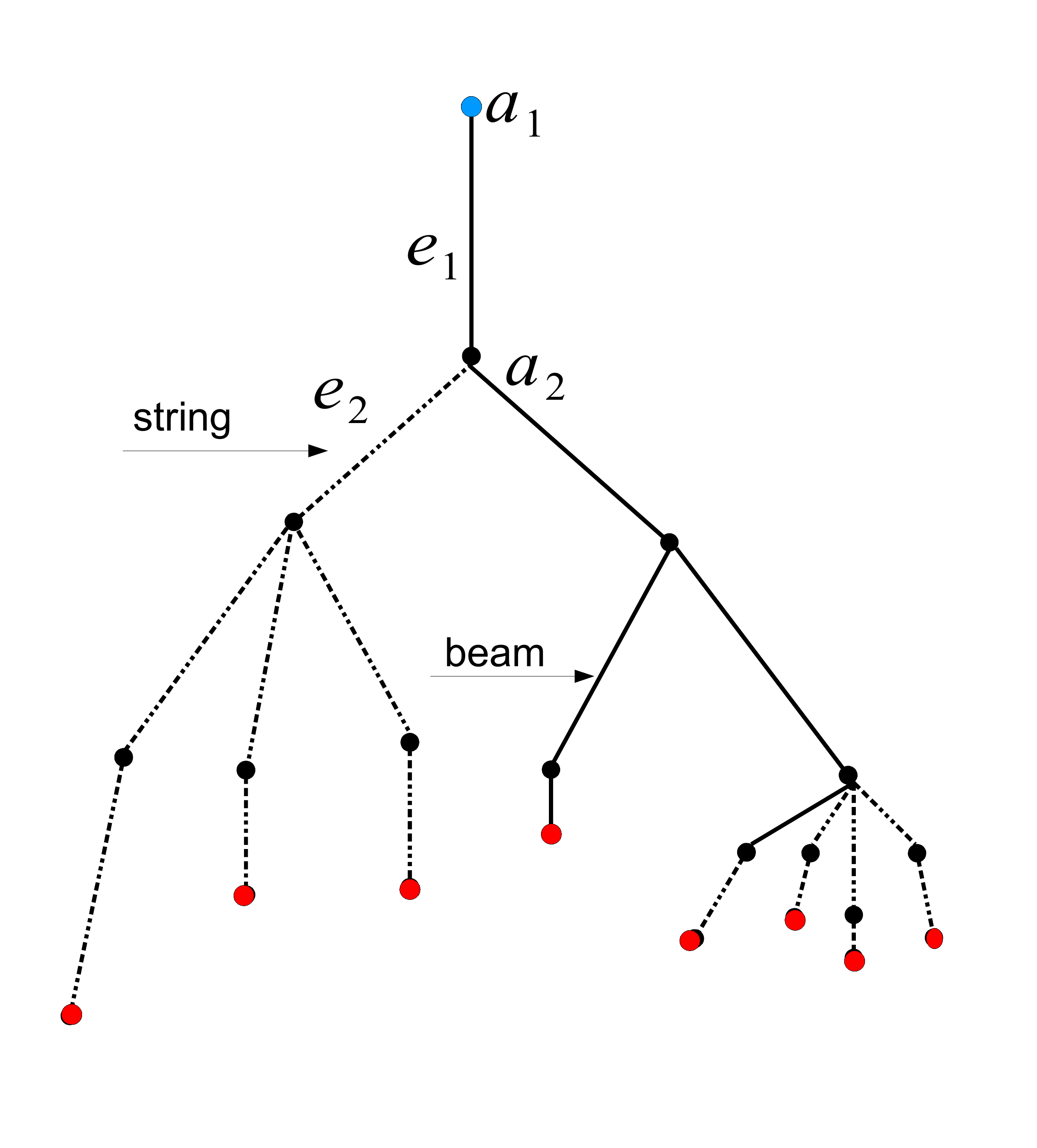,width=\linewidth}
  \caption{first tree \label{fig1}}
 \end{minipage} \hfill 
\begin{minipage}[b]{.48\linewidth}
  \centering\epsfig{figure=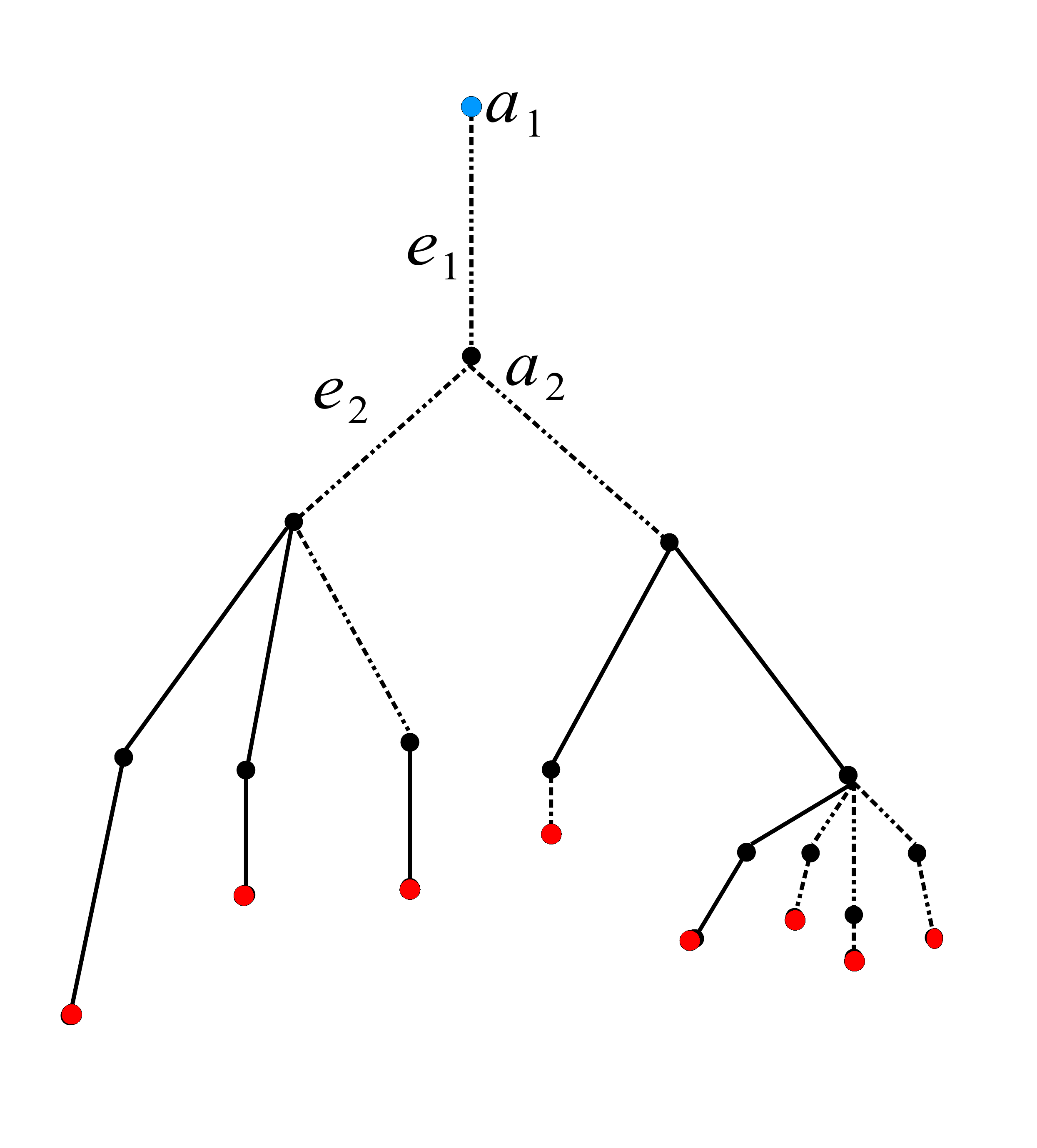,width=\linewidth}
  \caption{second~tree \label{fig2}}
 \end{minipage}
\end{figure}

We will use the following two results. The first gives us a necessary and
sufficient condition for the exponential stability of a $\mathcal{C}_{0}$%
-semigroup of contraction, for the proof see \cite{Gea78}, \cite{Hua85} or 
\cite{Pru84}.

\begin{theorem}
A $\mathcal{C}_{0}$-semigroup of contraction $e^{t\mathcal{L}}$ on a Hilbert
space is exponentially stable if, and only if, 
\begin{equation}
\mathbf{i}\mathbb{R}\cap \sigma (\mathcal{L})=\emptyset  \label{2.1}
\end{equation}
and 
\begin{equation}
\underset{\left| \beta \right| \rightarrow \infty }{\lim }\sup \left\| (%
\mathbf{i}\beta -\mathcal{L})^{-1}\right\| <\infty.  \label{2.2}
\end{equation}
\end{theorem}

The second, due to \cite{Tom10} (see also \cite{Bat08} and \cite{Liu05}),
characterizes the polynomial decay of a $\mathcal{C}_{0}$-semigroup of
contraction.

\begin{theorem}
A $\mathcal{C}_{0}$-semigroup of contraction $e^{t\mathcal{L}}$ on a Hilbert
space $\mathcal{H}$ satisfies 
\begin{equation*}
\left\| e^{t\mathcal{L}}y_{0}\right\| \leq \frac{C}{t^{\frac{1}{\alpha }}}%
\left\| y_{0}\right\| _{\mathcal{D}(\mathcal{L})}
\end{equation*}
for some constant $C>0$ and for $\alpha >0$ if, and only if, (\ref{2.1})
holds and 
\begin{equation}
\underset{\left| \beta \right| \rightarrow \infty }{\lim }\sup \frac{1}{%
\beta ^{\alpha }}\left\| (\mathbf{i}\beta -\mathcal{L})^{-1}\right\| <\infty.
\label{2.2'}
\end{equation}
\end{theorem}

\subsection{Asymptotic stability}

In this section, we prove (\ref{2.1}), that is, the system $(\mathcal{S})$
is asymptotically stable.

\begin{theorem}
The semigroup $(T(t))_{t \geq 0},$ generated by the operator $\mathcal{A}$
is asymptotically stable.
\end{theorem}

\begin{proof}
It suffices to show that (\ref{2.1}) holds, otherwise, by taking into
account Remark \ref{remarque}, there is a real number $\beta ,$ such that 
$\lambda :=i\beta $ is an eigenvalue of $\mathcal{A}.$ Let $y=(u,v)$ the
corresponding eigenvector. We have 
\begin{equation}
\left\{ 
\begin{tabular}{l}
$
\begin{tabular}{l}
$v^{j}=\lambda u^{j}$ \ \ \ \ \ \ for $j$ in $\{1,...,N\},$ \\ 
$\partial _{x}^{2}u^{j}=\lambda v^{j}$ \ \ for $j$ in $I_{S},$ \\ 
$-\partial _{x}^{4}u^{j}=\lambda v^{j}$ \ \ for $j$ in $I_{B}.$%
\end{tabular}
$%
\end{tabular}
\right.  \label{8.2}
\end{equation}

If $\lambda =0,$ multiplying the second and the third equations of (\ref{8.2}) by $u^{j}$ and summing, we obtain, using (\ref{s1'}) and (\ref{s2''}), 
\begin{equation*}
\sum_{j\in I_{S}}\left\| \partial _{x}u^{j}\right\| ^{2}+\sum_{j\in
I_{B}}\left\| \partial _{x}^{2}u^{j}\right\| ^{2}=0
\end{equation*}
which implies that $u=0$ and hence $v=0.$ Then, in the sequal, we suppose
that $\lambda \neq 0.$

Taking the real part of the inner product of $\lambda y-\mathcal{A}y=0$ with 
$y$ in $\mathcal{H},$ we obtain 
\begin{equation*}
Re(\left\langle \mathcal{A}y,y\right\rangle _{\mathcal{H}})=-\sum%
\limits_{a_{k}\in \mathcal{V}_{ext}^{\ast }}\left| v^{j_{k}}(a_{k})\right|
^{2}=0.
\end{equation*}

Thus $v^{j_{k}}(a_{k})=0$ for $a_{k}\in \mathcal{V}_{ext}^{\ast }$ and then $u^{j_{k}}(a_{k})=0$ for $a_{k}\in \mathcal{V}_{ext}^{\ast },$ which holds
for $k=1$; $\partial _{x}^{3}u^{j_{k}}(a_{k})=0$ for $a_{k}\in 
\mathcal{V}_{ext}^{B}$ and $\partial _{x}u^{j_{k}}(a_{k})=0$ for $a_{k}\in 
\mathcal{V}_{ext}^{S}.$

Then, $u$ is zero on every edge attached to a leaf, and by iteration, on every maximal subgraph of strings not followed by beams.

Now let $\mathcal{G}^{\prime }$ a maximal subgraph of beams not followed by strings. We want to prove that  $u$ is zero on $\mathcal{G}^{\prime }$. 

\textit{First case}: $\mathcal{G}^{\prime }=\mathcal{G}$. For each $j$ in $\{1,...,N\}$, substituting the first equation of (\ref{8.2})
into the third, we obtain, 
\begin{equation}
\partial _{x}^{4}u^{j}+\lambda ^{2}u^{j}=0.  \label{sys1}
\end{equation}

For the sequel, we use a matrix method \cite{Bel85}. First of all we shall
introduce some definitions and notations used in \cite{Bel85} (see also \cite
{Abd12}).

The adjacency matrix $E=(e_{jk})_{p\times p}$ of $\mathcal{G}$ is defined by 
\begin{equation*}
e_{jk}=\left\{ 
\begin{tabular}{l}
$1$ if $a_{j}$ and $a_{k}$ are adjacent, \\ 
$0$ otherwise.
\end{tabular}
\right.
\end{equation*}
The Hadamard product of two matrices $A=(a_{jk})_{p\times p}$ and $%
B=(b_{jk})_{p\times p}$ is the matrix $A\ast B=(a_{jk}b_{jk})_{p\times p}$.
For a function $r:\mathbb{R\rightarrow R},$ we define the matrix $%
r(A)=(r_{jk})_{p\times p}$ by 
\begin{equation*}
r_{jk}=\left\{ 
\begin{tabular}{l}
$r(a_{jk})$ if $e_{jk}=1$, \\ 
$0$ otherwise,
\end{tabular}
\right.
\end{equation*}
in particular, if $r(x)=x^{q}$ then, we write $A^{(q)}$ instead of $r(A).$
Furthermore, the matrix $L=(\ell _{jk})_{p\times p}$ is defined as follows 
\begin{equation*}
\ell _{jk}=\left\{ 
\begin{tabular}{l}
$\ell _{s(j,k)}$ if $e_{jk}=1$, \\ 
$0$ otherwise,
\end{tabular}
\right.
\end{equation*}
where $s(j,k)$ is the indice of the edge connecting $a_{j}$ to $a_{k}.$

To a function $f$ on $\mathcal{G}$ is associated the matrix function $F$
defined by 
\begin{equation*}
F:[0,1]\longrightarrow \mathbb{C}^{p\times p},x\longmapsto
F(x)=(f_{jk}(x))_{p\times p},
\end{equation*}
with 
\begin{equation*}
f_{jk}(x)=e_{jk}f_{s(j,k)}\left[ \ell _{s(j,k)}\left( \dfrac{1+d_{js(j,k)}}{2%
}-xd_{js(j,k)}\right) \right] .
\end{equation*}

The system (\ref{sys1}) is then rewritten as 
\begin{equation}
L^{(-4)}\ast U^{\prime \prime \prime \prime }+\lambda U=0.  \label{sys2}
\end{equation}

Integrate equation (\ref{sys2}) we obtain 
\begin{equation}
U(x)=A_{1}\ast \cos (\sqrt{\beta }Lx)+A_{2}\ast \sin (\sqrt{\beta }%
Lx)+B_{1}\ast \cosh (\sqrt{\beta }Lx)+B_{2}\ast \sinh (\sqrt{\beta }Lx),
\label{eq0}
\end{equation}
where without loss of generality we have supposed that $\beta >0$ and with $%
A_{1},A_{2},B_{1},B_{2}\in \mathbb{C}^{p\times p}.$ Then 
\begin{eqnarray}
L^{(-1)}\ast U^{\prime } &=&\sqrt{\beta }\left( -A_{1}\ast \sin (\sqrt{\beta }%
Lx)+A_{2}\ast \cos (\sqrt{\beta }Lx)\right.  \notag \\
&&\left. +B_{1}\ast \sinh (\sqrt{\beta }x)+B_{2}\ast \cosh (\sqrt{\beta }
Lx)\right) ,  \label{eq1} \\
L^{(-2)}\ast U^{\prime \prime } &=&\beta \left( -A_{1}\ast \cos (\sqrt{\beta }Lx)-A_{2}\ast
\sin (\sqrt{\beta }Lx)\right.  \notag \\
&&\left. +B_{1}\ast \cosh (\sqrt{\beta }x)+B_{2}\ast \sinh (\sqrt{\beta }
Lx)\right) ,  \label{eq2} \\
L^{(-3)}\ast U^{\prime \prime \prime } &=&\beta ^{3/2}\left( A_{1}\ast \sin (\sqrt{\beta }
Lx)-A_{2}\ast \cos (\sqrt{\beta }Lx)\right.  \notag \\
&&\left. + B_{1}\ast \sinh (\sqrt{\beta }x)+B_{2}\ast \cosh (\sqrt{\beta }
Lx)\right) .  \label{eq3}
\end{eqnarray}
The function $U$ satisfies also, 
\begin{equation}
U(1-x)=U(x)^{T}.  \label{h5}
\end{equation}

The boundary and transmission conditions can be expressed as follows:\newline
For the continuity condition of $u$ at the inner nodes, there exists $%
\varphi =\left( 
\begin{array}{c}
\varphi _{1} \\ 
\vdots  \\ 
\varphi _{p}
\end{array}
\right) \in \mathbb{C}^{p}$ such that 
\begin{equation}
U(0)=(\varphi e^{T})\ast E,  \label{h1}
\end{equation}
where $e=\left( 
\begin{array}{c}
1 \\ 
\vdots  \\ 
1
\end{array}
\right) \in \mathbb{R}^{p}.$

Since $u$ is zero at all external nodes then $\varphi _{j}=0$ when $a_{j}$
is an external node.

The continuity condition of $\partial _{x}^{2}u$ at the interior nodes and
the fact that $\partial _{x}^{2}u$ is zero at the root can be expressed in
this manner,

there exists $\psi =\left( 
\begin{array}{c}
\psi _{1} \\ 
\vdots  \\ 
\psi _{p}
\end{array}
\right) \in \mathbb{C}^{p}$ such that $\psi _{1}=0$, and 
\begin{equation}
L^{(-2)}\ast U^{\prime \prime }(0)=(\psi e^{T})\ast E.  \label{h2}
\end{equation}

The forth condition of (\ref{s1'}) and the fifth of (\ref{s2''}) applied to $u$ are expressed respectively as follows 
\begin{equation}
(L^{(-1)}\ast U^{\prime }(0)\ast E^{\ast })e=0  \label{h3}
\end{equation}
and 
\begin{equation}
(L^{(-3)}\ast U^{\prime \prime \prime }(0)\ast E^{\ast })e=0  \label{h4}
\end{equation}
where $E^{\ast }$ is obtained from $E$ by annulling the first line.

Substitute (\ref{h1}-\ref{h2}) in (\ref{eq1}-\ref{eq3}), and (\ref{h3}-\ref
{h4}) in (\ref{eq0}-\ref{eq2}) leads to 
\begin{eqnarray}
A_{1} &=&\frac{1}{2}(U(0)-L^{(-2)}\ast U^{\prime \prime }(0))=\frac{1}{2}%
((\varphi -\psi )e^{T})\ast E,  \label{h6} \\
B_{1} &=&\frac{1}{2}(U(0)+L^{(-2)}\ast U^{\prime \prime }(0))=\frac{1}{2}%
((\varphi +\psi )e^{T})\ast E,  \label{h7}
\end{eqnarray}
and 
\begin{eqnarray}
(A_{2}\ast E^{\ast })e &=&0,  \label{h8} \\
(B_{2}\ast E^{\ast })e &=&0.  \label{h9}
\end{eqnarray}

By taking $x=1$ in (\ref{eq0}) and (\ref{eq2}) and using (\ref{h5}) we get,
by combining the two obtained equations 
\begin{equation}
\sinh ^{(-1)}(\sqrt{\beta }L)\ast (B_{1}^{T}-B_{1}\cosh (\sqrt{\beta }
L))=B_{2}.  \label{eq5_1}
\end{equation}
Multiplying, in the Hadamard product, the above equation by $E^{\ast },$ we
get, using (\ref{h9}) 
\begin{equation}
\left( \sinh ^{(-1)}(\sqrt{\beta }L)\ast (B_{1}^{T}-B_{1}\cosh (\sqrt{\beta }
L))\ast E^{\ast }\right) e=0,  \label{eq5}
\end{equation}
We recall the following elementary rules for a matrix $M\in \mathbb{C}
^{p\times p}$ (see \cite{Dek00}), 
\begin{equation}
(M\ast B_{1}^{T})e=M(\varphi +\psi ),\text{ \ \ }(M\ast B_{1})e=diag(Me)(\varphi
+\psi ).  \label{77}
\end{equation}
Then, (\ref{eq5}) implies 
\begin{equation*}
J(\varphi +\psi )=0,
\end{equation*}
where 
\begin{equation*}
J=\sinh ^{(-1)}(\sqrt{\beta }L)\ast E^{\ast }-diag\left[ \left( \sinh
^{(-1)}(\sqrt{\beta }L)\ast \cosh (\sqrt{\beta }L)\ast E^{\ast }\right) e%
\right].
\end{equation*}

The matrix, obtained from $J\ast E^{\ast T}\ast E^{\ast }$ by removing rows
and colums that are zero, is a strictly diagonally dominant matrix. Since $
\varphi _{1}=\psi _{1}=0,$ this implies that the vector $\varphi +\psi ,$
and hence the matrix $B_{1},$ is zero. Return to (\ref{eq5_1}) we deduce that $
B_{2}=0.$

For $j=1,...,N,$ the expression of $u^{j}$ is then 
\begin{equation*}
u^{j}(x)=a_{1}^{j}\cos (\sqrt{\beta }x)+a_{2}^{j}\sin (\sqrt{\beta }%
x),\;\;(a_{1}^{j},a_{2}^{j}\in \mathbb{C})
\end{equation*}
which easily implies, using the transmissions conditions and the fact that $%
u $ and $\partial _{x}u$ vanish at the leaves, that 
\begin{equation*}
u=0.
\end{equation*}

\textit{Second case}: $\mathcal{G}^{\prime }\neq \mathcal{G}$. Let $a^{\prime }$ be the nearest node of $\mathcal{G}^{\prime }$ to $a_{1}$. Then $a^{\prime }$
is an end of at least one string.

For simplicity of notations we will suppose, in this part, that $a^{\prime }=a_{1}$  and $\mathcal{G}
^{\prime }=\mathcal{G}$ but with boundary conditions at $a_{1}$:
\begin{equation*}
\left\{ 
\begin{tabular}{l}
$u^{j}(a_{1})=u^{l}(a_{1})\;\ j,l\in I(a_{1}),$ \\ 
$\partial _{x}^{2}u^{j}(a_{1})=\partial _{x}^{2}u^{l}(a_{1})\;\ j,l\in I(a_{1}),$ \\ 
$\sum_{j\in I_{B}(a_{1})}d_{1j}\partial _{x}u^{j}(a_{1})=0\ .$
\end{tabular}
\right. 
\end{equation*}

As for the first case, there is $\psi =\left( 
\begin{array}{c}
\psi _{1} \\ 
\vdots  \\ 
\psi _{p}
\end{array}
\right) $ in $\mathbb{C}^{p}$ with $\varphi _{j}=0$ when $a_{j}\in \mathcal{V
}_{ext}^{\ast }$, such that 
\begin{equation*}
U(0)=(\varphi e^{T})\ast E
\end{equation*}
and 
\begin{equation*}
L^{(-2)}\ast U^{\prime \prime }(0)=(\psi e^{T})\ast E.
\end{equation*}
The third and forth conditions of (\ref{s1'}) and the fifth of (\ref{s2''}) applied to $
u$ are expressed as follows :
\begin{equation*}
(L^{(-1)}\ast U^{\prime }(0)\ast E)e=0
\end{equation*}
and 
\begin{equation*}
(L^{(-3)}\ast U^{\prime \prime \prime }(0)\ast E^{\ast })e=0.
\end{equation*}

As in the first case we obtain (\ref{h6}), (\ref{h7}), (\ref{h8}), (\ref{h9}) and (\ref{eq5_1}). Moreover, \\ $(\varphi _{k}+\psi _{k})_{k=2,...,p}$ will
be the trivial solution of a homogeneous linear system whose matrix is
invertible. Then $B_{1}\ast E^{\ast }$ and $B_{2}\ast E^{\ast T}\ast E^{\ast
}$ are zero.

For $j\in \{1,...,N\}-I(a_{1})$ the expression of $u^{j}$ is then 
\begin{equation*}
u^{j}(x)=a_{1}^{j}\cos (\sqrt{\beta }x)+a_{2}^{j}\sin (\sqrt{\beta }
x),\;\;(a_{1}^{j},a_{2}^{j}\in \mathbb{C}),
\end{equation*}
which easily implies, by using the transmissions conditions and the fact
that $u$ and $\partial _{x}u$ vanish at the leaves, 
\begin{equation*}
u^{j}=0.
\end{equation*}
Then we can suppose that $\mathcal{G}^{\prime }$ is a start of beams $e_{j},$
$j\in I(a_{1}).$ Without loss of generality, we identify $e_{j}$ with $%
(0,\ell _{j})$ by taking $\pi _{j}(0)=a_{1}.$ In such case we have the
following system : 
\begin{equation*}
\left\{ 
\begin{tabular}{l}
$u^{j}(\ell _{j})=0,\;j\in I(a_{1}),$ \\ 
$\partial _{x}u^{j}(\ell _{j})=0$ and $\partial _{x}^{3}u^{j}(\ell
_{j})=0,\;j\in I(a_{1}),$ \\ 
$u^{j}(0)=u^{l}(0)$ and $\partial _{x}^{2}u^{j}(0)=\partial
_{x}^{2}u^{l}(0),\;j,l\in I(a_{1}),$ \\ 
$\sum\limits_{j} \partial _{x}u^{j}(0)=0.$
\end{tabular}
\right.
\end{equation*}
Which implies 
\begin{equation*}
\left\{ 
\begin{tabular}{l}
$a_{1}^{j}\cos (\sqrt{\beta }\ell _{j})+a_{2}^{j}\sin (\sqrt{\beta }\ell
_{j})+b_{1}^{j}\cosh (\sqrt{\beta }\ell _{j})+b_{2}^{j}\sinh (\sqrt{\beta }
\ell _{j})=0,\;j\in I(a_{1}),$ \\ 
$-a_{1}^{j}\sin (\sqrt{\beta }\ell _{j})+a_{2}^{j}\cos (\sqrt{\beta }\ell
_{j})=0,\;j\in I(a_{1}),$ \\ 
$b_{1}^{j}\sinh (\sqrt{\beta }\ell _{j})+b_{2}^{j}\cosh (\sqrt{\beta }\ell
_{j})=0,\;j\in I(a_{1}),$ \\ 
$b_{1}^{j}=b_{1}^{l},\;j,l\in I(a_{1}),$ \\ 
$\sum\limits_{j} a_{2}^{j}+b_{2}^{j}=0.$%
\end{tabular}
\right.
\end{equation*}
The discriminant of the above system is 
\begin{equation*}
\Delta =\sum\limits_{j}\left( \prod\limits_{k\neq j}\cosh (\sqrt{\beta }\ell
_{k})(\sin (\sqrt{\beta }\ell _{j})+\sinh (\sqrt{\beta }\ell _{j}))\right)
\end{equation*}
which is different from zero. We conclude that $u^{j}=0,\;j\in I(a_{1}).$ That is, $u$
is null on $\mathcal{G}^{\prime }$.

By iteration, and using transmission conditions we conclude that $u$ is zero on $\mathcal{G}$. 
The above discussion is sufficient to conclude that $y=0,$ which contradicts
the fact that $y\neq 0.$
\end{proof}

\subsection{Exponential stability}

In this section we suppose that there are no beam following a string (from the root to leaves), that is to say  on every tree branch there is no beam between a string and a leaf (Figure 
\ref{fig1}). We prove that the solution of the whole system $(\mathcal{S}%
)$ is exponentially stable.

\begin{theorem}
\label{th1}\label{theo}If there are no beam following a string ( from the
root to leaves) then, the system $(\mathcal{S})$ is exponentially stable.
\end{theorem}

\begin{proof}
It suffices to prove that (\ref{2.2}) holds. Suppose the conclusion is
false. Then there exists a sequence $(\beta _{n})$ of real numbers, without
loss of generality, with $\beta _{n}\longrightarrow +\infty $, and a
sequence of vectors $(y_{n})=(u_{n},v_{n})$ in $\mathcal{D}(\mathcal{A})$
with $\left\| y_{n}\right\| _{\mathcal{H}}=1$, such that 
\begin{equation*}
\left\| (\mathbf{i}\beta _{n}I-\mathcal{A})y_{n}\right\| _{\mathcal{H}
}\longrightarrow 0
\end{equation*}
which is equivalent to 
\begin{eqnarray}
\mathbf{i}\beta _{n}u_{n}^{j}-v_{n}^{j} &=&f_{n}^{j}\longrightarrow
0,\;\;\;in\;H^{1}(0,\ell _{j}),\;\;j\ \text{in }I_{S},  \label{8.6} \\
\mathbf{i}\beta _{n}u_{n}^{j}-v_{n}^{j} &=&f_{n}^{j}\longrightarrow
0,\;\;\;in\;H^{2}(0,\ell _{j}),\;\;j\ \text{in }I_{B}, \\
\mathbf{i}\beta _{n}v_{n}^{j}-\partial _{x}^{2}u_{n}^{j}
&=&g_{n}^{j}\longrightarrow 0,\;\;\;in\;L^{2}(0,\ell _{j}),\;\;j\ \text{in }%
I_{S}, \\
\mathbf{i}\beta _{n}v_{n}^{j}+\partial _{x}^{4}u_{n}^{j}
&=&g_{n}^{j}\longrightarrow 0,\;\;\;in\;L^{2}(0,\ell _{j}),\;\;j\ \text{in }
I_{B}.  \label{8.7}
\end{eqnarray}
Then 
\begin{eqnarray}
-\beta _{n}^{2}u_{n}^{j}-\partial _{x}^{2}u_{n}^{j} &=&g_{n}^{j}+\mathbf{i}%
\beta _{n}f_{n}^{j},\;\;j\ \text{in }I_{S},  \label{2.13} \\
-\beta _{n}^{2}u_{n}^{j}+\partial _{x}^{4}u_{n}^{j} &=&g_{n}^{j}+\mathbf{i}%
\beta _{n}f_{n}^{j},\;\;j\ \text{in }I_{B}  \label{2.13'}
\end{eqnarray}
and 
\begin{equation*}
\left\| v_{n}^{j}\right\| ^{2}-\beta _{n}^{2}\left\| u_{n}^{j}\right\|
^{2}\longrightarrow 0,\;\;j=1,...,N.
\end{equation*}

First, Since 
\begin{equation*}
Re(\left\langle (\mathbf{i}\beta _{n}-\mathcal{A})y_{n},y_{n}\right\rangle _{%
\mathcal{H}})=\sum\limits_{a_{k}\in \mathcal{V}_{ext}^{\ast }}\left|
v^{j_{k}}(a_{k})\right| ^{2},
\end{equation*}
we obtain 
\begin{equation*}
\left| v_{n}^{j_{k}}(a_{k})\right| \longrightarrow 0,\;\text{for }j\in 
\mathcal{V}_{ext}^{\ast }.
\end{equation*}
Then $\beta _{n}u_{n}^{j_{k}}(a_{k})\longrightarrow 0$ for $a_{k}\in 
\mathcal{V}_{ext}^{\ast },$ $\partial
_{x}^{3}u_{n}^{j_{k}}(a_{k})\longrightarrow 0$ for $a_{k}\in \mathcal{V}%
_{ext}^{B}$ and $\partial _{x}u_{n}^{j_{k}}(a_{k})\longrightarrow 0$ for $%
a_{k}\in \mathcal{V}_{ext}^{S},$ and recall that $\partial
_{x}u_{n}^{j_{k}}(a_{k})=0$ for $j\in \mathcal{V}_{ext}^{B}$.

Let $a_{k}\in \mathcal{V}_{ext}^{S}$ and $q$ a function in $C^{1}([0,\ell
_{j_{k}}],\mathbb{C})$. The real part of the inner product of (\ref{2.13})
with $q\partial _{x}u_{n}^{j_{k}}$ gives 
\begin{gather}
\left. \frac{1}{2}\beta _{n}^{2}\left| u_{n}^{j_{k}}(x)\right|
^{2}q(x)\right| _{0}^{\ell _{j_{k}}}+\left. \frac{1}{2}\left| \partial
_{x}u_{n}^{j_{k}}(x)\right| ^{2}q(x)\right| _{0}^{\ell _{j_{k}}}  \notag \\
-\frac{1}{2}\int_{0}^{\ell _{j_{k}}}\left( \left| \partial
_{x}u_{n}^{j_{k}}(x)\right| ^{2}+\beta _{n}^{2}\left|
u_{n}^{j_{k}}(x)\right| ^{2}\right) \partial_{x}q(x)dx  \notag \\
+\left. Re\left( i\beta _{n}f_{n}^{j_{k}}(x)q(x)\overline{u_{n}^{j_{k}}}%
(x)\right) \right| _{0}^{\ell _{j_{k}}}\longrightarrow 0.  \label{pr}
\end{gather}
With $q(x)=x$ or $q(x)=\ell _{j_{k}}-x$ we can deduce easily that 
\begin{equation}
-\frac{1}{2}\int_{0}^{\ell _{j_{k}}}\left( \left| \partial
_{x}u_{n}^{j_{k}}(x)\right| ^{2}+\beta _{n}^{2}\left|
u_{n}^{j_{k}}(x)\right| ^{2}\right) dx\longrightarrow 0,  \label{norm}
\end{equation}
and as in \cite{Far12} it follows that, 
\begin{equation}
\beta _{n}u_{n}^{j_{k}}(a_{s})\longrightarrow 0,\;\;\partial
_{x}u_{n}^{j_{k}}(a_{s})\longrightarrow 0,\;\text{and }Re\left( \mathbf{i}%
\beta _{n}f_{n}^{j_{k}}(a_{s})\overline{u_{n}^{j_{k}}}(a_{s})\right)
\longrightarrow 0  \label{s11}
\end{equation}
where $a_{s}$ is the end of $e_{j_{k}}$ different from $a_{k}.$

\noindent By iteration we conclude that for every $j\in I_{S}-\{1\},$ the
properties (\ref{norm}) and (\ref{s11}) hold.

If there is no beam in the tree, then the fifth condition in (\ref{s2''})
with (\ref{s11}) imply that $\partial _{x}u_{n}^{1}(a_{2})\longrightarrow 0.$
Then as for $j$ in $\mathcal{V}_{ext}^{S},$ with using (\ref{s11}) again and
the continuity condition of $u$ at internal nodes, we obtain 
\begin{equation*}
-\frac{1}{2}\int_{0}^{\ell _{1}}\left( \left| \partial
_{x}u_{n}^{1}(x)\right| ^{2}+\beta _{n}^{2}\left| u_{n}^{1}(x)\right|
^{2}\right) dx\longrightarrow 0.
\end{equation*}
Then, we conclude that $\left\| y_{n}\right\| \rightarrow 0$ which
contradicts the fact that $\left\| y_{n}\right\| =1$ and the proof is then
complete.

Now suppose that there is at least one beam. Without loss of generality, we
suppose that there is no string in $\mathcal{G}.$

Let $j$ in $\{1,...,N\}$ and $q$ a function in $C^{2}([0,\ell _{j}],\mathbb{C%
})$ such that $\partial _{x}^{2}q=0.$ We want to calculate the real part of
the inner product of (\ref{2.13'}) by $q\partial _{x}u_{n}^{j}.$

Straight-forward calculations give 
\begin{eqnarray*}
&&Re\left( \left\langle -\beta _{n}^{2}u_{n}^{j},q\partial
_{x}u_{n}^{j}\right\rangle \right) +Re\left( \left\langle \partial
_{x}^{4}u_{n}^{j},q\partial _{x}u_{n}^{j}\right\rangle \right) = \\
&&-\frac{1}{2}\beta _{n}^{2}\left. \left| u_{n}^{j}(x)\right|
^{2}q(x)\right| _{0}^{\ell _{j}}+\frac{1}{2}\int_{0}^{\ell _{j}}\beta
_{n}^{2}\left| u_{n}^{j}\right| ^{2}\partial _{x}qdx+Re\left( \left.
\partial _{x}^{3}u_{n}^{j}(x)q(x)\partial _{x}\overline{u_{n}^{j}}(x)\right|
_{0}^{\ell _{j}}\right) \\
&&-\frac{1}{2}\left. \left| \partial _{x}^{2}u_{n}^{j}(x)\right|
^{2}q(x)\right| _{0}^{\ell _{j}}+\frac{3}{2}\int_{0}^{\ell _{j}}\left|
\partial _{x}^{2}u_{n}^{j}\right| ^{2}\partial _{x}qdx-Re\left( \left.
\partial _{x}^{2}u_{n}^{j}(x)\partial _{x}\overline{u_{n}^{j}}\partial
_{x}q(x)\right| _{0}^{\ell _{j}}\right) ,
\end{eqnarray*}
and
$$
Re\left( \left\langle g_{n}^{j}+\mathbf{i}\beta _{n}f_{n}^{j},q\partial
_{x}u_{n}^{j}\right\rangle \right) 
= Re\left( \int_{0}^{\ell _{j}}g_{n}^{j}\partial _{x}\overline{u_{n}^{j}}
qdx\right) 
$$
$$
-Re\left( \mathbf{i}\beta _{n}\int_{0}^{\ell _{j}}\partial
_{x}(qf_{n}^{j})\overline{u_{n}^{j}}dx\right) +Re\left. \mathbf{i}\beta
_{n}f_{n}^{j}(x)q(x)\overline{u_{n}^{j}}(x)\right| _{0}^{\ell _{j}}.
$$

Since $g_{n}^{j},$ $f_{n}^{j}$ and $\partial _{x}(qf_{n}^{j})$ converge to $%
0 $ and $\mathbf{i}\beta _{n}u_{n}^{j}$ and $\partial _{x}u_{n}^{j}$ are
bounded, the first and the second terms of the right member of the above
equality converge to $0.$ It follows 
\begin{gather}
-\frac{1}{2}\beta _{n}^{2}\left. \left| u_{n}^{j}(x)\right| ^{2}q(x)\right|
_{0}^{\ell _{j}}+Re\left( \left. \partial _{x}^{3}u_{n}^{j}(x)q(x)\partial
_{x}\overline{u_{n}^{j}}(x)\right| _{0}^{\ell _{j}}\right) -\frac{1}{2}
\left. \left| \partial _{x}^{2}u_{n}^{j}(x)\right| ^{2}q(x)\right|
_{0}^{\ell _{j}}  \notag \\
-Re\left( \left. \partial _{x}^{2}u_{n}^{j}(x)\partial _{x}\overline{
u_{n}^{j}}\partial _{x}q(x)\right| _{0}^{\ell _{j}}\right) -Re\left. \left( 
\mathbf{i}\beta _{n}f_{n}^{j}(x)q(x)\overline{u_{n}^{j}}(x)\right) \right|
_{0}^{\ell _{j}}  \notag \\
+\frac{1}{2}\int_{0}^{\ell _{j}}\beta _{n}^{2}\left| u_{n}^{j}\right|
^{2}\partial _{x}qdx + 
\frac{3}{2}\int_{0}^{\ell _{j}}\left| \partial
_{x}^{2}u_{n}^{j}\right| ^{2}\partial _{x}qdx\longrightarrow 0.
\label{2.13'''}
\end{gather}

In particular if $a_{k}$ is in $\mathcal{V}_{ext}^{B}$ then with $j=j_{k}$
and $q(x)=x$ or $q(x)=\ell _{j_{k}}-x$ (\ref{2.13'''}) becomes 
$$
-\frac{\ell _{j_{k}}}{2}\left| \partial _{x}^{2}u_{n}^{j_{k}}(a_{k})\right|
^{2}+Re\left( d_{kj_{k}}\partial
_{x}^{2}u_{n}^{j_{k}}(a_{s})\partial _{x}\overline{u_{n}^{j_{k}}}
(a_{s})\right) +\frac{1}{2}\int_{0}^{\ell _{j_{k}}}\beta _{n}^{2}\left|
u_{n}^{j_{k}}\right| ^{2}dx +
$$
\begin{equation}
\frac{3}{2}\int_{0}^{\ell _{j_{k}}}\left|
\partial _{x}^{2}u_{n}^{j_{k}}\right| ^{2}dx\rightarrow 0  \label{2.13''2}
\end{equation}
where $a_{s}$ is the end of $e_{j_{k}}$ different from $a_{k}.$

Multiplying (\ref{2.13'}) by $\frac{1}{\beta _{n}^{1/2}}e^{-\beta
_{n}^{1/2}(\ell _{j_{k}}-x)}$ or by $\frac{1}{\beta _{n}^{1/2}}e^{-\beta
_{n}^{1/2}x},$ then, as in \cite{Far13}, after noting that $\partial
_{x}^{3}u_{n}^{j_{k}}(a_{k})$ and $\beta _{n}u_{n}^{j_{k}}(a_{k})$ tend to $0,$ we obtain 
\begin{equation}
\partial _{x}^{2}u_{n}^{j_{k}}(a_{k})\longrightarrow 0.  \label{exp1}
\end{equation}
Hence, (\ref{2.13''2}) can be rewritten as 
\begin{equation}
Re\left( d_{kj_{k}}\partial
_{x}^{2}u_{n}^{j_{k}}(a_{s})\partial _{x}\overline{u_{n}^{j_{k}}}%
(a_{s})\right) +\frac{1}{2}\int_{0}^{\ell _{j_{k}}}\beta _{n}^{2}\left|
u_{n}^{j_{k}}\right| ^{2}dx+\frac{3}{2}\int_{0}^{\ell _{j_{k}}}\left|
\partial _{x}^{2}u_{n}^{j_{k}}\right| ^{2}dx\longrightarrow 0.
\label{2.13''}
\end{equation}

Now let $q\equiv 1$ (\ref{2.13'''}) can be rewritten, for $j_{k},$ as 
$$
-\frac{1}{2}\beta _{n}^{2}\left| u_{n}^{j_{k}}(a_{s})\right| ^{2}+Re\left(
\partial _{x}^{3}u_{n}^{j_{k}}(a_{s})\partial _{x}\overline{u_{n}^{j_{k}}}
(a_{s})\right) -\frac{1}{2}\left| \partial
_{x}^{2}u_{n}^{j_{k}}(a_{s})\right| ^{2} - 
$$
$$
Re\left( \mathbf{i}\beta
_{n}f_{n}^{j_{k}}(a_{s})\overline{u_{n}^{j_{k}}}(a_{s})\right) \rightarrow 0.
$$

For $j\ $in $I(a_{s}),$ multiplying (\ref{2.13'}) by $\frac{1}{\beta
_{n}^{1/2}}e^{-\beta _{n}^{1/2}x}$ or by $\frac{1}{\beta _{n}^{1/2}}%
e^{-\beta _{n}^{1/2}(\ell _{j}-x)}$, we get, 
\begin{equation*}
\frac{1}{\beta _{n}^{1/2}}d_{sj}\partial
_{x}^{3}u_{n}^{j}(a_{s})+\varepsilon \partial _{x}^{2}u_{n}^{j}(a_{s})+\beta
_{n}^{1/2}d_{sj}\partial _{x}u_{n}^{j}(a_{s})+\varepsilon \beta
_{n}u_{n}^{j}(a_{s})\longrightarrow 0.
\end{equation*}
with $\varepsilon \in \{-1,1\}$.
Summing over $j\in I(a_{s})$, by taking into account the continuity
condition of $u_{n}$ and $\partial _{x}^{2}u_{n},$ the damping
conditions, four in (\ref{s1'}) and fifth in (\ref{s2''}), and boundary condition at $a_{1}$ if $j=1$, we deduce 
\begin{equation*}
\partial _{x}^{2}u_{n}^{j}(a_{s})+\beta _{n}u_{n}^{j}(a_{s})\longrightarrow
0,
\end{equation*}
which leads to 
\begin{equation*}
\frac{1}{\beta _{n}^{1/2}}\partial _{x}^{3}u_{n}^{j}(a_{s})+\beta
_{n}^{1/2}\partial _{x}u_{n}^{j}(a_{s})=\alpha _{n}^{j}\longrightarrow 0.
\end{equation*}
Hence, for any positive real number $a$ we have 
\begin{equation*}
Re\left( \partial _{x}^{3}u_{n}^{j}(a_{s})\partial _{x}\overline{u_{n}^{j}}%
(a_{s})\right) 
\begin{tabular}[t]{l}
$=Re\left( \frac{\partial _{x}^{3}u_{n}^{j}(a_{s})}{\beta _{n}^{1/2}}\beta
_{n}^{1/2}\partial _{x}\overline{u_{n}^{j}}(a_{s})\right) $ \\ 
$=Re\left( (\alpha _{n}^{j}-\beta _{n}^{1/2}\partial
_{x}u_{n}^{j}(a_{s}))\beta _{n}^{1/2}\partial _{x}\overline{u_{n}^{j}}%
(a_{s})\right) $ \\ 
$\leq -\beta _{n}\left| \partial _{x}u_{n}^{j}(a_{s})\right| ^{2}+\frac{a}{2}%
\beta _{n}\left| \partial _{x}u_{n}^{j}(a_{s})\right| ^{2}+\frac{1}{2a}%
\left| \alpha _{n}^{j}\right| ^{2}.$
\end{tabular}
\end{equation*}
Moreover, for any real positive number $b$ we have 
\begin{equation*}
-Re\left( \mathbf{i}\beta _{n}f_{n}^{j}(a_{s})\overline{u_{n}^{j}}%
(a_{s})\right) \leq \frac{b}{2}\beta _{n}^{2}\left| u_{n}^{j}(a_{s})\right|
^{2}+\frac{1}{2b}\left| f_{n}^{j}(a_{s})\right| ^{2}.
\end{equation*}
Thus, we obtain the following framing for $j_{k},$ 
\begin{gather*}
-\frac{1}{2}\beta _{n}^{2}\left| u_{n}^{j_{k}}(a_{s})\right| ^{2}+Re\left(
\partial _{x}^{3}u_{n}^{j_{k}}(a_{s})\partial _{x}\overline{u_{n}^{j_{k}}}
(a_{s})\right) -\frac{1}{2}\left| \partial
_{x}^{2}u_{n}^{j_{k}}(a_{s})\right| ^{2} \\
-Re\left( \mathbf{i}\beta _{n}f_{n}^{j_{k}}(a_{s})\overline{u_{n}^{j_{k}}}
(a_{s})\right) -\frac{1}{2b}\left| f_{n}^{j_{k}}(a_{s})\right| ^{2} \\
\leq (-\frac{1}{2}+\frac{b}{2})\beta _{n}^{2}\left|
u_{n}^{j_{k}}(a_{s})\right| ^{2}-\frac{1}{2}\left| \partial
_{x}^{2}u_{n}^{j_{k}}(a_{s})\right| ^{2}+(-1+\frac{a}{2})\beta _{n}\left|
\partial _{x}u_{n}^{j_{k}}(a_{s})\right| ^{2}\leq 0
\end{gather*}
with $a=b=\frac{1}{2}.$ Which implies that 
\begin{equation*}
\beta _{n}^{2}\left| u_{n}^{j_{k}}(a_{s})\right| ^{2},\;\left| \partial
_{x}^{2}u_{n}^{j_{k}}(a_{s})\right| ^{2},\text{ }\beta _{n}\left| \partial
_{x}u_{n}^{j_{k}}(a_{s})\right| ^{2},\text{ }\frac{1}{\beta _{n}}\left|
\partial _{x}^{3}u_{n}^{j_{k}}(a_{s})\right| ^{2}\longrightarrow 0
\end{equation*}
and then all the expressions 
\begin{equation*}
Re\left( \partial _{x}^{3}u_{n}^{j_{k}}(a_{s})\partial _{x}\overline{
u_{n}^{j_{k}}}(a_{s})\right) ,\;Re\left( \partial
_{x}^{2}u_{n}^{j_{k}}(a_{s})\partial _{x}\overline{u_{n}^{j_{k}}}%
(a_{s})\right) \text{ and }Re\left( \mathbf{i}\beta _{n}f_{n}^{j_{k}}(a_{s})%
\overline{u_{n}^{j_{k}}}(a_{s})\right) 
\end{equation*}
tend to $0$ as $n$ goes to infinity. Then (\ref{2.13''}) leads to 
\begin{equation*}
\frac{1}{2}\int_{0}^{\ell _{j_{k}}}\beta _{n}^{2}\left| u_{n}^{j_{k}}\right|
^{2}dx+\frac{3}{2}\int_{0}^{\ell _{j_{k}}}\left| \partial
_{x}^{2}u_{n}^{j_{k}}\right| ^{2}dx\longrightarrow 0.
\end{equation*}

We iterate such procedure to obtain that for $j\neq 1,$
\begin{equation*}
\frac{1}{2}\int_{0}^{\ell _{j}}\beta _{n}^{2}\left| u_{n}^{j}\right| ^{2}dx+
\frac{3}{2}\int_{0}^{\ell _{j}}\left| \partial _{x}^{2}u_{n}^{j}\right|
^{2}dx\longrightarrow 0
\end{equation*}
and for $j\in I(a_{2})-\{1\},$ 
\begin{equation*}
\beta _{n}^{2}\left| u_{n}^{j}(a_{2})\right| ^{2},\;\left| \partial
_{x}^{2}u_{n}^{j}(a_{2})\right| ^{2},\text{ }Re\left( \partial
_{x}^{3}u_{n}^{j}(a_{2})\partial _{x}\overline{u_{n}^{j}}(a_{2})\right)
,\;Re\left( \partial _{x}^{2}u_{n}^{j}(a_{2})\partial _{x}\overline{u_{n}^{j}%
}(a_{2})\right) \rightarrow 0.
\end{equation*}

Finally let $j=1.$ Using continuity conditions of $u_{n}$ and $\partial
_{x}^{2}u_{n},$ and damping conditions, four in (\ref{s1'}) and fifth in (
\ref{s2''}), (\ref{2.13'''}) leads to 
\begin{equation*}
\frac{1}{2}\int_{0}^{\ell _{1}}\beta _{n}^{2}\left| u_{n}^{1}\right| ^{2}dx+%
\frac{3}{2}\int_{0}^{\ell _{1}}\left| \partial _{x}^{2}u_{n}^{1}\right|
^{2}dx\longrightarrow 0.
\end{equation*}
In conclusion $\left\| y_{n}\right\| $ converge to $0,$ which contradicts the
hypothesis that $\left\| y_{n}\right\| =1.$
\end{proof}

\subsection{Polynomial stability}

In this section we suppose that there is at least a beam following a string 
(Figure \ref{fig2}). We will prove that the solution of the whole system $(%
\mathcal{S})$ is polynomially stable and not exponentially stable.

\begin{theorem}
\label{th2}If at least one beam follows a string from the root to a leaf (which is equivelent to say that the beam is located between two strings or a string and a leaf ) then, the $\mathcal{C}_{0}$%
-semigroup $(T(t)_{t\geq 0}$ is polynomially stable. More precisely: \newline
if all the sets of beams following strings are singletons then there is $C>0$
such that 
\begin{equation*}
\left\| e^{t\mathcal{A}}y_{0}\right\| \leq \frac{C}{t}\left\| y_{0}\right\|
_{\mathcal{D}(\mathcal{A})}
\end{equation*}
for every $y_{0}\in \mathcal{D}(\mathcal{A}).$\newline
If there is a string followed by at least two beams then there is $C>0$ such
that 
\begin{equation*}
\left\| e^{t\mathcal{A}}y_{0}\right\| \leq \frac{C}{t^{2/3}}\left\|
y_{0}\right\| _{\mathcal{D}(\mathcal{A})}
\end{equation*}
for every $y_{0}\in \mathcal{D}(\mathcal{A}).$
\end{theorem}

\begin{proof}
In the sequel $\alpha $ is equal to $1$ or $3/2.$

It suffices to prove that (\ref{2.2'}) holds. Suppose the conclusion is
false. Then there exists a sequence $(\beta _{n})$ of real numbers, without
loss of generality, with $\beta _{n}\longrightarrow +\infty ,$ and a
sequence of vectors $(y_{n})=(u_{n},v_{n})$ in $\mathcal{D}(\mathcal{A})$
with $\left\| y_{n}\right\| _{\mathcal{H}}=1$, such that 
\begin{equation*}
\left\| \beta _{n}^{\alpha }(\mathbf{i}\beta _{n}I-\mathcal{A})y_{n}\right\|
_{\mathcal{H}}\longrightarrow 0
\end{equation*}
which is equivalent to 
\begin{eqnarray}
\beta _{n}^{\alpha }(\mathbf{i}\beta _{n}u_{n}^{j}-v_{n}^{j})
&=&f_{n}^{j}\longrightarrow 0,\;\;\;\text{in}\;H^{1}(0,\ell _{j}),\;\;j\ 
\text{in }I_{S},  \label{8.61} \\
\beta _{n}^{\alpha }(\mathbf{i}\beta _{n}u_{n}^{j}-v_{n}^{j})
&=&f_{n}^{j}\longrightarrow 0,\;\;\;\text{in}\;H^{2}(0,\ell _{j}),\;\;j\ 
\text{in }I_{B}, \\
\beta _{n}^{\alpha }(\mathbf{i}\beta _{n}v_{n}^{j}-\partial
_{x}^{2}u_{n}^{j}) &=&g_{n}^{j}\longrightarrow 0,\;\;\;\text{in}%
\;L^{2}(0,\ell _{j}),\;\;j\ \text{in }I_{S}, \\
\beta _{n}^{\alpha }(\mathbf{i}\beta _{n}v_{n}^{j}+\partial
_{x}^{4}u_{n}^{j}) &=&g_{n}^{j}\longrightarrow 0,\;\;\;\text{in}%
\;L^{2}(0,\ell _{j}),\;\;j\ \text{in }I_{B}.  \label{8.71}
\end{eqnarray}
Then 
\begin{eqnarray}
-\beta _{n}^{\alpha }(\beta _{n}^{2}u_{n}^{j}+\partial _{x}^{2}u_{n}^{j})
&=&g_{n}^{j}+\mathbf{i}\beta _{n}f_{n}^{j},\;\;j\ \text{in }I_{S},
\label{2.1311} \\
\beta _{n}^{\alpha }(-\beta _{n}^{2}u_{n}^{j}+\partial _{x}^{4}u_{n}^{j})
&=&g_{n}^{j}+\mathbf{i}\beta _{n}f_{n}^{j},\;\;j\ \text{in }I_{B},
\label{2.1312}
\end{eqnarray}
and 
\begin{equation*}
\left\| v_{n}^{j}\right\| ^{2}-\beta _{n}^{2}\left\| u_{n}^{j}\right\|
^{2}\longrightarrow 0,\;\;j=1,...,N.
\end{equation*}

Since $Re(\left\langle \beta _{n}^{\alpha }(\mathbf{i}\beta _{n}-\mathcal{A}%
)y_{n},y_{n}\right\rangle _{\mathcal{H}})=\sum\limits_{a_{k}\in \mathcal{V}%
_{ext}^{\ast }}\beta _{n}^{\alpha }\left| v_{n}^{j_{k}}(a_{k})\right| ^{2},$
we obtain 
\begin{equation*}
\beta _{n}^{\frac{\alpha }{2}}\left| v_{n}^{j_{k}}(a_{k})\right|
\longrightarrow 0,\;\text{for }a_{k}\in \mathcal{V}_{ext}^{\ast }.
\end{equation*}
Then $\forall a_{k}\in \mathcal{V}_{ext}^{B},$ $\beta _{n}^{\frac{\alpha }{2}%
}\partial _{x}^{3}u_{n}^{j_{k}}(a_{k})\longrightarrow 0,$ $\forall a_{k}\in 
\mathcal{V}_{ext}^{S}$ $\beta _{n}^{\frac{\alpha }{2}}\partial
_{x}u_{n}^{j_{k}}(a_{k})\longrightarrow 0,$ $\forall a_{k}\in \mathcal{V}%
_{ext}^{\ast }$, $\beta _{n}^{1+\frac{\alpha }{2}}u_{n}^{j_{k}}(a_{k})%
\longrightarrow 0,$ and recall that $\partial _{x}u_{n}^{j_{k}}(a_{k})=0$
for $a_{k}\in \mathcal{V}_{ext}^{B}$.

As in the proof of the previous theorem, if $e_{j}$ is a string followed by
no beam then 
\begin{equation}
-\frac{1}{2}\beta _{n}^{\alpha }\int_{0}^{\ell _{j}}\left( \left| \partial
_{x}u_{n}^{j}(x)\right| ^{2}+\beta _{n}^{2}\left| u_{n}^{j}(x)\right|
^{2}\right) dx\longrightarrow 0,  \label{norm1}
\end{equation}
and if $a_{k}$ is an end of such edge then 
\begin{equation}
\beta _{n}^{1+\frac{\alpha }{2}}u_{n}^{j}(a_{k})\longrightarrow 0\text{ and }%
\beta _{n}^{\frac{\alpha }{2}}\partial _{x}u_{n}^{j}(a_{k})\longrightarrow 0.
\label{s111}
\end{equation}

Now we can suppose without loss of generality that all edges are beams,
except that related to the root which is a string.

Again, as in the previous proof, we have that for every beam $e_{j}$ which
not adjacent to the string, 
\begin{equation}
\beta _{n}^{1+\frac{\alpha }{2}}\left\| u_{n}^{j}\right\| ,\;\beta _{n}^{%
\frac{\alpha }{2}}\left\| \partial _{x}^{2}u_{n}^{j}\right\| \longrightarrow
0,  \label{for}
\end{equation}
and if $a_{k}$ is an end of $e_{j}$ then 
\begin{equation}
\beta _{n}^{1+\frac{\alpha }{2}}\left| u_{n}^{j}(a_{k})\right|
\longrightarrow 0,\;\beta _{n}^{\frac{\alpha }{2}}\left| \partial
_{x}^{2}u_{n}^{j}(a_{k})\right| \longrightarrow 0,\text{ }  \label{s1111}
\end{equation}
and all the expressions 
\begin{equation*}
Re\left( \beta _{n}^{\alpha }\partial _{x}^{3}u_{n}^{j}(a_{k})\partial _{x}%
\overline{u_{n}^{j}}(a_{k})\right) ,\;Re\left( \beta _{n}^{\alpha }\partial
_{x}^{2}u_{n}^{j}(a_{k})\partial _{x}\overline{u_{n}^{j}}(a_{k})\right) 
\text{ and }Re\left( \mathbf{i}\beta _{n}^{1+\alpha }f_{n}^{j}(a_{k})%
\overline{u_{n}^{j}}(a_{k})\right)
\end{equation*}
tend to zero as $n$ goes to infinity.

Let $j\in I(a_{2})-\{1\}$ (equivalently, $e_{j}$ is a beam attached to $%
a_{2} $), then the real part of the inner product of (\ref{2.1312}) by $%
q\partial _{x}u_{n}^{j},$ with $q=x$ or $q=\ell _{j}-x$ gives 
$$
Re\left( \beta _{n}^{\alpha }d_{kj}\partial
_{x}^{2}u_{n}^{j}(a_{2})\partial _{x}\overline{u_{n}^{j}}(a_{2})\right) +%
\frac{1}{2}\beta _{n}^{\alpha }\int_{0}^{\ell _{j}}\beta _{n}^{2}\left|
u_{n}^{j}\right| ^{2}dx + 
$$
\begin{equation}
\frac{3}{2}\beta _{n}^{\alpha }\int_{0}^{\ell
_{j}}\left| \partial _{x}^{2}u_{n}^{j}\right| ^{2}dx\longrightarrow 0
\label{2.p}
\end{equation}

Summing over $j$ in $I(a_{2})-\{1\},$ one can deduce easily that 
\begin{equation*}
\sum_{j\in I(a_{2})-\{1\}}\left( \frac{1}{2}\beta _{n}^{\alpha
}\int_{0}^{\ell _{j}}\beta _{n}^{2}\left| u_{n}^{j}\right| ^{2}dx+\frac{3}{2}%
\beta _{n}^{\alpha }\int_{0}^{\ell _{j}}\left| \partial
_{x}^{2}u_{n}^{j}\right| ^{2}dx\right) \longrightarrow 0
\end{equation*}
and in particular, for every beam $e_{j}$%
\begin{equation*}
\beta _{n}^{\alpha }\left( \int_{0}^{\ell _{j}}\beta _{n}^{2}\left|
u_{n}^{j}\right| ^{2}dx+\int_{0}^{\ell _{j}}\left| \partial
_{x}^{2}u_{n}^{j}\right| ^{2}dx\right) \longrightarrow 0
\end{equation*}

Now we want prove that, for $j$ in $I(a_{2})-\{1\},$ $\beta
_{n}u_{n}^{j}(a_{2})$ and $\partial _{x}^{3}u_{n}^{j}(a_{2})$ converge to $0$%
.

\textit{First case:} In this case there is one and only one beam $e_{j}$
attached to $e_{1}$ and $\alpha =1.$

The real part of the inner product of (\ref{2.1312}) by $q\partial
_{x}u_{n}^{j},$ with $q=1$ gives 
\begin{equation}
-\frac{1}{2}\beta _{n}^{3}\left| u_{n}^{j}(a_{2})\right| ^{2}-\frac{1}{2}%
\beta _{n}\left| \partial _{x}^{2}u_{n}^{j}(a_{2})\right| ^{2}\rightarrow 0,
\label{tt}
\end{equation}
then $\beta _{n}^{3}\left| u_{n}^{j}(a_{2})\right| ^{2}$ and $\beta
_{n}\left| \partial _{x}^{2}u_{n}^{j}(a_{2})\right| $ converge to zero as $n$
goes to infinite.

Multiplying (\ref{2.1312}) by $\frac{1}{\beta _{n}}e^{-\beta _{n}^{1/2}x}$
or by $\frac{1}{\beta _{n}}e^{-\beta _{n}^{1/2}(\ell _{j}-x)}$, we obtain,
using the results below 
\begin{equation*}
\partial _{x}^{3}u_{n}^{j}(a_{2})\longrightarrow 0.
\end{equation*}
We deduce, using the fifth condition in (\ref{s2''}), that $\partial
_{x}u_{n}^{1}(a_{1})\longrightarrow 0.$

\textit{Second case:} In this case there is a string followed by at least
two beams and $\alpha =3/2.$ 
First, (\ref{2.1312}) implies 
\begin{equation}
\frac{\left\| \partial _{x}^{4}u_{n}^{j}\right\| }{\beta _{n}^{1/4}}%
\longrightarrow 0.  \label{st}
\end{equation}
Second, we need the following lemma (due to Gagliardo and Nirenberg \cite{Liu99}):
\end{proof}

\begin{lemma}
\label{5}

\begin{enumerate}
\item  There are two positive constants $C_{1}$ and $C_{2}$ such that for
any $w$ in $H^{1}(0,\ell _{j}),$%
\begin{equation}
\left\| w\right\| _{\infty }\leq C_{1}\left\| \partial _{x}w\right\|
^{1/2}\left\| w\right\| ^{1/2}+C_{2}\left\| w\right\| .  \label{g1}
\end{equation}

\item  There are two positive constants $C_{3}$ and $C_{4}$ such that for
any $w $ in $H^{2}(0,\ell _{j}),$%
\begin{equation}
\left\| \partial _{x}w\right\| \leq C_{3}\left\| \partial _{x}^{2}w\right\|
^{1/2}\left\| w\right\| ^{1/2}+C_{4}\left\| w\right\| .  \label{g2}
\end{equation}
\end{enumerate}
\end{lemma}

\begin{proof}[The remainder of the proof]
Applying the previous lemma several times, one can deduce easily that 
\begin{equation}
\partial
_{x}^{3}u_{n}^{j}(a_{1})\longrightarrow 0.  \label{s1111'}
\end{equation}

Return back to the fifth condition in (\ref{s2''}) with using (\ref{s1111'}), we
get $\partial _{x}u_{n}^{1}(a_{1})\longrightarrow 0.$ Now multiplying (\ref
{2.1311}), when $u^{j}=u^{1}$, with $q(x)=\ell _{1}-x\ $or $q(x)=x,$ then as
for (\ref{pr}) and (\ref{norm}), with using (\ref{s1111'}) again and the
continuity condition of $u$ at $a_{2}$, it follows 
\begin{equation*}
-\frac{1}{2}\int_{0}^{\ell _{1}}\left( \left| \partial
_{x}u_{n}^{1}(x)\right| ^{2}+\beta _{n}^{2}\left| u_{n}^{1}(x)\right|
^{2}\right) dx\longrightarrow 0.
\end{equation*}
Then we conclude that $\left\| y_{n}\right\| \rightarrow 0$ which contradicts
the fact that $\left\| y_{n}\right\| =1$ and the proof is then complete.
\end{proof}

Now we consider a reduced system composed of one string $e_{1}$ and one beam 
$e_{2}$ such that $\ell _{1}=\ell _{2}=\pi $ and with control is applied on
the beam. Precisely we consider the system

\begin{equation*}
(\mathcal{S}_{0}):\left\{ 
\begin{tabular}{l}
$u_{tt}^{1}-u_{xx}^{1}=0$ in $(0,\pi )\times (0,\infty ),$ \\ 
$u_{tt}^{2}+u_{xxxx}^{2}=0$ in $(0,\pi )\times (0,\infty ),$ \\ 
\\ 
$u^{1}(0,t)=u^{2}(0,t),\;u_{x}^{2}(0,t)=0,\;u_{xxx}^{2}(0,t)=u_{x}^{1}(0,t),$
\\ 
$u^{1}(\pi ,t)=0,\;u_{xxx}^{2}(\pi ,t)=u_{t}^{2}(\pi ,t),\;u_{x}^{2}(\pi
,t)=0,$ \\ 
\\ 
$u^{j}(x,0)=u_{0}^{j}(x),\;\;u_{t}^{j}(x,0)=u_{1}^{j}(x),\;j=1,2.$%
\end{tabular}
\right.
\end{equation*}
In view of Theorem \ref{th2} the system $(\mathcal{S}_{0})$ is polynomial
stable and we will prove that it is not exponentially stable. Note that if
the control is applied on the string instead of the beam then the system is
the exponential stability (by Theorem \ref{th1}).

\begin{theorem}
The system $(\mathcal{S}_{0})$ is not exponential stable in the energy space 
$\mathcal{H}$.
\end{theorem}

\begin{proof}
We prove that the corresponding semigroup $(T(t))_{t \geq 0}$ is not
exponentially stable.

For $n\in \mathbb{N},$ such that $\sqrt{n}$ is integer and even let $\beta
_{n}=n^{2}+2\sqrt{n}+\frac{1}{n}$ and $f_{n}=(0,0,-\sin \beta _{n}x,0),$
then $\beta _{n}\rightarrow +\infty $ and $f_{n}$ is in $\mathcal{H}$ and is
bounded. Let $y_{n}=(u_{n}^{1},u_{n}^{2},v_{n}^{1},v_{n}^{2})\in \mathcal{D}(
\mathcal{A})$ such that $(\mathcal{A}-i\beta _{n})y_{n}=f_{n}.$ We will
prove that $y_{n}\rightarrow +\infty.$

We have 
\begin{eqnarray}
\beta _{n}^{2}u_{n}^{1}+\partial _{x}^{2}u_{n}^{1} &=&\sin \beta _{n}x,
\label{ll1} \\
-\beta _{n}^{2}u_{n}^{2}+\partial _{x}^{4}u_{n}^{2} &=&0.  \label{ll2}
\end{eqnarray}
then $u_{n}^{1}$ and $u_{n}^{2}$ are of the form 
\begin{eqnarray*}
u_{n}^{1} &=&c_{1}\sin (\beta _{n}x)+(-\frac{x}{2\beta _{n}}+c_{2})\cos
(\beta _{n}x), \\
u_{n}^{2} &=&d_{1}\sin (\sqrt{\beta _{n}}x)+d_{2}\cos (\sqrt{\beta _{n}}%
x)+d_{3}\sinh (\sqrt{\beta _{n}}x)+d_{4}\cosh (\sqrt{\beta _{n}}x).
\end{eqnarray*}
The transmission and boundary conditions are rewritten as follows 
\begin{eqnarray}
d_{2}+d_{4} &=&c_{2},  \label{l1} \\
\sqrt{\beta _{n}}(d_{1}+d_{3}) &=&0,  \label{l2} \\
\beta _{n}^{3/2}(-d_{1}+d_{3}) &=&-\frac{1}{2\beta _{n}}+\beta _{n}c_{1},
\label{l3}
\end{eqnarray}
and 
\begin{eqnarray}
c_{1}\sin (\beta _{n}\pi )+(-\frac{\pi }{2\beta _{n}}+c_{2})\cos (\beta
_{n}\pi ) &=&0,  \label{l4} \\
d_{1}\cos (\sqrt{\beta _{n}}\pi )-d_{2}\sin (\sqrt{\beta _{n}}\pi
)+d_{3}\cosh(\sqrt{\beta _{n}}\pi )+d_{4}\sinh(\sqrt{\beta _{n}}\pi ) &=&0,
\label{l5}
\end{eqnarray}
and 
\begin{eqnarray}
&&\beta _{n}^{3/2}(-d_{1}\cos (\sqrt{\beta _{n}}\pi )+d_{2}\sin (\sqrt{\beta
_{n}}\pi )+d_{3}\cosh (\sqrt{\beta _{n}}\pi )+d_{4}\sinh (\sqrt{\beta _{n}}
\pi ))  \notag \\
&=&\mathbf{i}\beta _{n}(d_{1}\sin (\sqrt{\beta _{n}}\pi )+d_{2}\cos (\sqrt{%
\beta _{n}}\pi )+d_{3}\sinh (\sqrt{\beta _{n}}\pi )+d_{4}\cosh (\sqrt{\beta
_{n}}\pi )).  \label{l6}
\end{eqnarray}

Summing (\ref{l5}) and (\ref{l6}), we obtain 
\begin{eqnarray}
&&2\sqrt{\beta _{n}}(d_{3}\cosh (\sqrt{\beta _{n}}\pi )+d_{4}\sinh (\sqrt{
\beta _{n}}\pi ))  \notag \\
&=&\mathbf{i}(d_{1}\sin (\sqrt{\beta _{n}}\pi )+d_{2}\cos (\sqrt{\beta _{n}}
\pi )+d_{3}\sinh (\sqrt{\beta _{n}}\pi )+d_{4}\cosh (\sqrt{\beta _{n}}\pi )).
\label{99}
\end{eqnarray}
Substituting (\ref{l1}) and (\ref{l2}) into (\ref{l5}) leads to 
\begin{equation}
d_{4}=\frac{\cosh (\sqrt{\beta _{n}}\pi )-\cos (\sqrt{\beta _{n}}\pi )}{
\sinh (\sqrt{\beta _{n}}\pi +\sin (\sqrt{\beta _{n}}\pi )}d_{1}+\frac{\sin (
\sqrt{\beta _{n}}\pi )}{\sinh (\sqrt{\beta _{n}}\pi +\sin (\sqrt{\beta _{n}}
\pi )}c_{2}.  \label{l8}
\end{equation}
Now, by substituting (\ref{l1}-\ref{l3}) and (\ref{l8}) into (\ref{99}) we
get 
\begin{eqnarray}
&&2d_{1}\left( \sqrt{\beta _{n}}h(\sqrt{\beta _{n}}\pi )-2\beta _{n}\tan
(\beta _{n}\pi )\sin (\sqrt{\beta _{n}}\pi )\sinh (\sqrt{\beta _{n}}\pi
)\right.  \notag\\
&&\left. +\mathbf{i}\left( 1-\cosh (\sqrt{\beta _{n}}\pi )\cos (\sqrt{\beta
_{n}}\pi )+\sqrt{\beta _{n}}\tan (\beta _{n}\pi )h(\sqrt{\beta _{n}}\pi
)\right) \right)  \notag \\
&=&\left( 2\sqrt{\beta _{n}}\sin (\sqrt{\beta _{n}}\pi )\sinh (\sqrt{\beta
_{n}}\pi )-\mathbf{i}h(\sqrt{\beta _{n}}\pi )\right) \left( -\frac{1}{2\beta
_{n}^{2}}\tan (\beta _{n}\pi )+\frac{\pi }{2\beta _{n}}\right)  \label{aa}
\end{eqnarray}
with 
\begin{equation*}
h(\sqrt{\beta _{n}}\pi )=\cosh (\sqrt{\beta _{n}}\pi )\sin (\sqrt{\beta _{n}}
\pi )+\sinh (\sqrt{\beta _{n}}\pi )\cos (\sqrt{\beta _{n}}\pi ).
\end{equation*}

We want to prove that $\beta _{n}^{3/2}d_{1}$ is equivalent to $\sqrt{n}\frac{
\pi ^{2}}{2}$ as $n$ goes to infinity.

Since $\beta _{n}=n^{2}+2\sqrt{n}+\frac{1}{n}$ then $\sqrt{\beta _{n}}=n+
\frac{1}{\sqrt{n}}=n(1+o(\frac{1}{\sqrt{n}}))$ and 
\begin{eqnarray*}
\sin (\sqrt{\beta _{n}}\pi ) &=&\sin (\frac{\pi }{\sqrt{n}})=\frac{\pi }{
\sqrt{n}}+o(\frac{1}{\sqrt{n}}), \\
\cos (\sqrt{\beta _{n}}\pi ) &=&\cos (\frac{\pi }{\sqrt{n}})=1+o(\frac{1}{
\sqrt{n}}), \\
\tan (\beta _{n}\pi ) &=&\tan (\frac{\pi }{n})=\frac{\pi }{n}+o(\frac{1}{n
\sqrt{n}}), \\
\cosh (\sqrt{\beta _{n}}\pi ) &=&\cosh ((n+\frac{1}{\sqrt{n}})\pi )=\frac{
e^{n\pi }}{2}(1+\frac{\pi }{\sqrt{n}}+o(\frac{1}{\sqrt{n}})), \\
\sinh (\sqrt{\beta _{n}}\pi ) &=&\frac{e^{n\pi }}{2}(1+\frac{\pi }{\sqrt{n}}
+o(\frac{1}{\sqrt{n}})),
\end{eqnarray*}
then 
\begin{eqnarray*}
h(\sqrt{\beta _{n}}\pi ) &=&\frac{e^{n\pi }}{2}(1+\frac{2\pi }{\sqrt{n}}+o(%
\frac{1}{\sqrt{n}})), \\
\sqrt{\beta _{n}}\sinh (\sqrt{\beta _{n}}\pi )\sin (\sqrt{\beta _{n}}\pi
)\tan (\beta _{n}\pi ) &=&\frac{e^{n\pi }}{2}(\frac{\pi ^{2}}{\sqrt{n}}+o(
\frac{1}{\sqrt{n}})) \\
\sqrt{\beta _{n}}\sinh (\sqrt{\beta _{n}}\pi )\sin (\sqrt{\beta _{n}}\pi )
&=&\pi \sqrt{n}\frac{e^{n\pi }}{2}(1+\frac{\pi }{\sqrt{n}}+o(\frac{1}{\sqrt{n%
}}))
\end{eqnarray*}
Hence, (\ref{aa}) implies
\begin{equation*}
2d_{1}\sqrt{\beta _{n}}\frac{e^{n\pi }}{2}\left( 1+o(1)\right) =\frac{%
e^{n\pi }}{2}\sqrt{n}(2\pi +o(1))(-\frac{1}{2\beta _{n}^{2}}\tan (\beta
_{n}\pi )+\frac{\pi }{2\beta _{n}}).
\end{equation*}
which leads to 
\begin{equation}
2\beta _{n}^{3/2}d_{1}\sim _{+\infty }\pi ^{2}\sqrt{n}.  \label{r1}
\end{equation}

Now suppose that $\partial _{x}u_{n}^{1}$ is bounded. The real part of the inner product of (\ref{ll1}) with $(\pi -x)\partial _{x}u_{n}^{1}$ gives 
\begin{equation*}
-\frac{\pi }{2}\left| -\frac{1}{2\beta _{n}}+\beta _{n}c_{1}\right| ^{2}-
\frac{\pi }{2}\left| \beta _{n}c_{2}\right| ^{2}=-\frac{1}{2}(\left\|
u_{n}^{1}\right\| ^{2}+\left\| \partial _{x}u_{n}^{1}\right\|
^{2})+Re(\int_{0}^{\pi }\sin (\beta _{n}x)(\pi -x)\partial _{x}\overline{u_{n}^{1}}
dx).
\end{equation*}
By taking into account (\ref{l2}-\ref{l3}) and (\ref{r1}), $\beta _{n}^{2}\left\| u_{n}^{1}\right\|
^{2}+\left\| \partial _{x}u_{n}^{1}\right\| ^{2}$ converge to infinity. In
conclusion $y_{n}$ is not bounded.
\end{proof}

\begin{remark}
Let $\varepsilon >0.$ By taking $\beta _{n}=n^{2}+2n^{1-\alpha }+\frac{1}{
n^{2\alpha }}$ with $0<\alpha <\varepsilon $ and such that $n^{1-\alpha }$
is integer and even and $y_{n}$ is such that $f_{n}=(\beta _{n}^{\frac{1}{2}%
-\varepsilon }(\mathcal{A}-i\beta _{n}))y_{n},$ then we can prove that $%
y_{n} $ is not bounded and then the polynomial stability of $(\mathcal{S})$
can't be better than $\frac{1}{t^{2}}.$
\end{remark}

\section*{Comment}

If we replace the boundary conditions by the followings 
\begin{eqnarray*}
u^{1}(a_{1},t) &=&0,\;(1-\delta )u_{xx}^{1}(a_{1},t)=0, \\
u_{x}^{j_{k}}(a_{k},t) &=&-u_{t}^{j_{k}}(a_{k},t),\;\;\;\;\;\;a_{k}\in 
\mathcal{V}_{ext}^{S}, \\
u_{xx}^{j_{k}}(a_{k},t)
&=&-u_{tx}^{j_{k}}(a_{k},t),\;\;u^{j_{k}}(a_{k},t)=0,\;\;\;\;\;\;\;a_{k}\in 
\mathcal{V}_{ext}^{B},
\end{eqnarray*}
then we obtain the same results.

\section*{Acknowledgements} The authors thank the referees for many valuable remarks which helped us to improve the paper significantly.


\end{document}